\documentclass{tran-l}
\usepackage{amsmath}
\usepackage{amssymb}
\usepackage{graphicx}
\usepackage[all,cmtip]{xy}
\RequirePackage{lineno}

 \newtheorem{theorem}{Theorem}[section]
 \newtheorem{corollary}[theorem]{Corollary}
 \newtheorem{lemma}[theorem]{Lemma}
 \newtheorem{proposition}[theorem]{Proposition}
 \theoremstyle{definition}
 \newtheorem{definition}[theorem]{Definition}
 \theoremstyle{remark}
 \newtheorem{remark}{Remark}[section]
 \theoremstyle{assumption}
 \newtheorem{assumption}{Assumption}[section]
 \theoremstyle{example}
 
 \newtheorem{bigthm}{Theorem}

\newtheorem{bigcor}[bigthm]{Corollary}

 \numberwithin{equation}{section}

\begin{document}

\title[Associativity of Gluing]
{On the Associativity of Gluing}

\author{LIZHEN QIN}

\address{Mathematics Department of Wayne State University, Detroit, MI 48202, USA}

\email{dv6624@wayne.edu}









\begin{abstract}
This paper studies the associativity of gluing of trajectories in Morse
theory. We show that the associativity of gluing follows from
of the existence of compatible manifold with face structures on the compactified moduli
spaces. Using our previous work, we obtain the
associativity of gluing in certain cases.

In particular, associativity holds when the ambient manifold is compact
and the vector field is Morse-Smale.
\end{abstract}
\maketitle

\section{Introduction}
In order to develop his homology theory, Floer invented two techniques in
Morse theory (see e.g. \cite{floer1}). One is the compactification
of the moduli spaces of negative gradient trajectories. The other
one is the gluing of broken trajectories. These two arguments have
continuously impacted Morse theory since then. For example,
moduli spaces have extensive applications in geometry and topology
(see e.g. \cite{franks}, \cite{latour},
\cite{barraud_cornea1}-\cite{burghelea_haller} and
\cite{cohen1}-\cite{cornea5}).

Due to this influence, there is a folklore theorem or rather a
philosophy as follows. Under certain conditions of compactness,
a moduli space of trajectories can be
compactified to be a manifold with corners.
There has been some progress on this topic in the
literature as it was interpreted and proved in certain cases. For
example, see \cite[Proposition 2.11]{latour}, \cite[Theorem
1]{burghelea_haller}, \cite[Appendix]{barraud_cornea2},
\cite[Theorem 3.3]{qin1} and \cite[Theorem 7.5]{qin2}.

Another related problem is the so-called
``associativity of gluing'' that is alluded to in the title.
We first learned of this problem  in the paper of Cohen, Jones and Segal \cite{CJS}.

This paper shows that the associativity of gluing is a direct consequence of
the existence of compatible manifold structures on the compactified moduli
spaces. We will in fact see that there is a general result along these lines
in which Morse theory occurs as a special case.

Suppose $p_{1}$, $p_{2}$ and $p_{3}$ are critical points,
$\gamma_{1}$ is a trajectory from $p_{1}$ to $p_{2}$ and
$\gamma_{2}$ is a trajectory from $p_{2}$ to $p_{3}$. In a strict
sense, the pair $(\gamma_{1}, \gamma_{2})$ of consecutive
trajectories is not a trajectory. We consider $(\gamma_{1},
\gamma_{2})$ as a broken trajectory from $p_{1}$ to $p_{3}$. A
{\it gluing} of $(\gamma_{1}, \gamma_{2})$ is a smooth family
$$\gamma_{1} \#_{\lambda} \gamma_{2}\, ,$$ where $\lambda \in
[0,\epsilon)$ is the gluing parameter,
$\gamma_1 \#_0 \gamma_2 = (\gamma_1,\gamma_2)$ and
$\gamma_{1} \#_{\lambda} \gamma_{2}$ is an unbroken trajectory when $\lambda \ne 0$.

Suppose now that $\gamma_{1}$, $\gamma_{2}$ and $\gamma_{3}$ are three
consecutive trajectories. Then one can form two families according
to the various ways of associating pairs:
$$ (\gamma_{1} \#_{\lambda_{1}} \gamma_{2})
\#_{\lambda_{2}} \gamma_{3} \qquad \text{ and  }\qquad  \gamma_{1} \#_{\lambda_{1}}
(\gamma_{2} \#_{\lambda_{2}} \gamma_{3})\, .$$
If these families coincide, one says that associativity gluing is
satisfied.

The manifold structure of a compactified moduli space is actually
related to the associativity of gluing. One can derive the manifold
structure from the associativity of gluing because the latter provides
nice coordinate charts for the former. However, this is \textit{not}
the only way to get the manifold structure. The papers
\cite{latour}, \cite{burghelea_haller}, \cite{qin1} and \cite{qin2}
do not use any gluing arguments.

In this paper, we shall strengthen the above relationship by working
in the opposite direction: we will show that
the associativity of gluing is a consequence of the existence of
a certain kind of manifold structure.
More precisely, Theorems \ref{theorem_gluing} and
\ref{theorem_tradition_gluing} show that, if the manifold structures
satisfy Assumption \ref{assumption_M_I}, then one will get the
associativity of gluing for free. In fact, we reformulate a gluing of
broken trajectories as parametrizations of collar neighborhoods of
the strata of the compactified moduli spaces. Then associativity of gluing
will be seen to be equivalent to a choice of compatible collar structure.
The above theorems will be generalized to Theorem \ref{theorem_collar}
which is a statement about the compatible collar structures of manifold with
faces.

In short, these theorems convert the problem of the associativity of
gluing to the problem of manifold structures. By the results we proved about
manifold structures in \cite{qin1} and \cite{qin2}, we get
Propositions \ref{proposition_infinite_gluing} and
\ref{proposition_compact_gluing}. They show the associativity of gluing
in  Morse theory in two contexts. An informal restatement of these
results is given by

\begin{bigcor} Suppose $M$ is a compact Riemannian manifold and $f$ is a Morse
function on $M$. Suppose $- \nabla f$
satisfies the Morse-Smale condition.

Then there exists an associative gluing rule.
\end{bigcor}

\begin{bigcor} Suppose $M$ is a complete Hilbert-Riemannian manifold. Assume $f$
satisfies Condition (C) and has finite indices.
Suppose $- \nabla f$
satisfies the Morse-Smale condition.
Assume that the metric on $M$ is locally trivial (see \cite[Definition
2.16]{qin1}).

Then there exists an associative gluing rule.
\end{bigcor}

A byproduct of our work is Proposition \ref{proposition_collar}
which is also about compatible collar structures. Theorem
\ref{theorem_collar} is about a family of manifolds with faces (see
Assumption \ref{assumption_set}), while Proposition
\ref{proposition_collar} is about a single one. However, the assumption
of Proposition \ref{proposition_collar} is more general.

The outline of this paper is as follows. Section
\ref{section_moduli_spaces} reviews the definition of moduli spaces
of trajectories. Section \ref{section_main_theorems} gives our main
results on the associativity of gluing. Section
\ref{section_generalization} generalizes the theorems in the
previous section. The proof of our main theorem occupies Sections
\ref{section_face_structures} and \ref{section_proof}. We conclude
this paper by presenting the byproduct in Section
\ref{section_a_byproduct}.

\section{Moduli Spaces}\label{section_moduli_spaces}
In this section, we review the definition of the moduli spaces of
trajectories of negative gradient vector fields. (See \cite{schwarz}
or \cite{qin1} for more details.)

Suppose $M$ is a Hilbert-Riemannian manifold and $f$ is a Morse function on
$M$. Let $-\nabla f$ be the negative gradient of $f$.

\begin{definition}\label{definition_invariant_manifold}
Let $\phi_{t}(x)$ be the flow generated by $- \nabla f$ with initial
value $x$. Suppose $p$ is a critical point. Define the descending
manifold of $p$ as $\mathcal{D}(p) = \{ x \in M \mid \displaystyle
\lim_{t \rightarrow - \infty} \phi_{t}(x) = p \}$. Define the
ascending manifold of $p$ as $\mathcal{A}(p) = \{ x \in M \mid
\displaystyle \lim_{t \rightarrow + \infty} \phi_{t}(x) = p \}$.
\end{definition}

Both $\mathcal{D}(p)$ and $\mathcal{A}(p)$ are smoothly embedded
submanifolds in $M$.

\begin{definition}\label{definition_transverality}
If the descending manifold $\mathcal{D}(p)$ and the ascending
manifold $\mathcal{A}(q)$ are transversal for all critical points
$p$ and $q$, then we say $- \nabla f$ satisfies the transversality
or Morse-Smale condition.
\end{definition}

If $- \nabla f$ satisfies transversality, then $\mathcal{D}(p) \cap
\mathcal{A}(q)$ is an embedded submanifold which consists of points
on trajectories (or flow lines) from $p$ to $q$. Since a trajectory
has an $\mathbb{R}$-action, we may take the quotient of
$\mathcal{D}(p) \cap \mathcal{A}(q)$ by this $\mathbb{R}$-action,
i.e. consider its orbit space acted upon by the flow. This leads to
the following definition.

\begin{definition}\label{definition_moduli_space}
Suppose $- \nabla f$ satisfies transversality. Define
$\mathcal{W}(p,q) = \mathcal{D}(p) \cap \mathcal{A}(q)$. Define the
moduli space $\mathcal{M}(p,q)$ as the orbit space
$\mathcal{W}(p,q)/\mathbb{R}$.
\end{definition}

We assume transversality all through this paper. It's well known
that, when $f$ has finite indices, $\mathcal{M}(p,q)$ is a finitely dimensional manifold of
dimension $\text{ind}(p) - \text{ind}(q) - 1$, where $\text{ind}(*)$
is the Morse index of $*$.

\begin{definition}\label{definition_points_partial_order}
Suppose $p$ and $q$ are two critical points. We define the relation
$p \succeq q$ if there is a trajectory from $p$ to $q$. We define
the relation $p \succ q$ if $p \succeq q$ and $p \neq q$.
\end{definition}

The transversality implies that
$``\succeq"$ is a partial order. To guarantee this, it suffices to
show the transitivity of $``\succeq"$. The best proof is probably to
use the $\lambda$-Lemma (see \cite[p.\ 85, Corollary 1]{palis_de}).
It is valid even if $M$ is a Banach manifold and the vector field is
a general one (\textit{not} necessarily a negative gradient) with
hyperbolic singularities. In Floer theory (see e.g. \cite[p.\
529]{floer1}), this can be proved by a gluing argument.

\begin{definition}\label{definition_critical_sequence}
An ordered set $I = \{ r_{0}, r_{1}, \cdots, r_{k+1} \}$ is a
critical sequence if $r_{i}$ ($i=0, \cdots, k+1$) are critical
points and $r_{0} \succ r_{1} \succ \cdots \succ r_{k+1}$. We call
$r_{0}$ the head of $I$, and $r_{k+1}$ the tail of $I$. The length
of $I$ is $|I|=k$.
\end{definition}

Suppose $I = \{ r_{0}, r_{1}, \cdots, r_{k+1} \}$ is a critical
sequence. We define the following product manifold
\begin{equation}\label{definition_M_I}
\mathcal{M}_{I} = \prod_{i=0}^{k} \mathcal{M}(r_{i}, r_{i+1}).
\end{equation}

Each element in $\mathcal{M}_{I}$ stands for a (un)broken trajectory
from $r_{0}$ to $r_{k+1}$ which is broken at exactly the points
$r_{i}$ ($i=1, \cdots, k$).

\section{Main Theorems}\label{section_main_theorems}
In this section, we state our results on the associativity of gluing.

Theorems \ref{theorem_gluing} and \ref{theorem_tradition_gluing}
will be based on the following assumption. For the definitions of
manifold with faces and the $k$-stratum, see Definitions
\ref{definition_manifold_with_face} and \ref{definition_k_stratum}.

\begin{assumption}\label{assumption_M_I}
Suppose $\Omega$ is the set of critical
points of $f$. Assume $\Omega$ is countable. The relation $``\succeq"$
(see Definition \ref{definition_points_partial_order}) defined on $\Omega$ is a
partial order. Suppose $\mathcal{M}(p,q)$ is a finite dimensional
manifold for each $p,q \in \Omega$ such that $p \succ q$ (see
Remark \ref{remark_smooth_1}). Suppose $\mathcal{M}(p,q)$ can be
compactified to $\overline{\mathcal{M}(p,q)}$ having the structure of a compact
smooth manifold with faces. In addition, assume each $\overline{\mathcal{M}(p,q)}$
satisfies the following conditions:

(1). We have $\overline{\mathcal{M}(p,q)} = \bigsqcup_{I}
\mathcal{M}_{I}$, where the disjoint union is over all critical
sequences with head $p$ and tail $q$. The $k$-stratum of
$\overline{\mathcal{M}(p,q)}$ is $\bigsqcup_{|I|=k}
\mathcal{M}_{I}$, and each $M_{I}$ is an open subset of the
$k$-stratum. The smooth structure of $\overline{\mathcal{M}(p,q)}$
is compatible with those of $\mathcal{M}_{I}$.

(2). Suppose $p \succ r \succ
q$, then the natural inclusion $\overline{\mathcal{M}(p,r)} \times
\overline{\mathcal{M}(r,q)} \hookrightarrow
\overline{\mathcal{M}(p,q)}$ is a smooth embedding.
\end{assumption}

\begin{remark}\label{remark_smooth_1}
By Definition
\ref{definition_moduli_space}, $\mathcal{M}(p,q)$ has a
natural smooth structure induced from those of $\mathcal{D}(p)$ and
$\mathcal{A}(q)$ (see e.g. \cite{schwarz}, \cite{latour},
\cite{burghelea_haller}, \cite{qin1} and \cite{qin2}). However, in
order to make Assumption \ref{assumption_M_I} hold, we may give
$\mathcal{M}(p,q)$ a smooth structure different from the above one
(see Remark \ref{remark_smooth_2}).
\end{remark}

In order to make the statement of gluing conceptual and strong, we
shall have to introduce the following formal definitions.

Suppose $I_{1} = \{ r_{0}, \cdots, r_{k+1} \}$ and $I_{2} = \{
r'_{0}, \cdots, r'_{l+1} \}$ are two critical sequences. If $I_{2}
\subseteq I_{1}$, $r'_{0}=r_{0}$ and $r'_{l+1}=r_{k+1}$, i.e. $I_{2}
= \{ r_{0}, r_{i_{1}}, \cdots, r_{i_{l}}, r_{k+1} \}$, denote them
by $I_{2} \preceq I_{1}$.

We use the notation $\Lambda_{I_{1}}$ to represent the gluing
parameter for $\mathcal{M}_{I_{1}}$. Here $\Lambda_{I_{1}} =
(\lambda_{1}, \cdots, \lambda_{|I_{1}|}) \in
\prod_{i=1}^{|I_{1}|}[0, +\infty) = [0, +\infty)^{|I_{1}|}$. By the
relation between $I_{1}$ and $I_{2}$, we introduce the following
definitions of the tuples induced from $\Lambda_{I_{1}}$. Define
$\Lambda_{I_{1},I_{2}} \in [0, +\infty)^{|I_{2}|}$ as
\begin{equation}
\Lambda_{I_{1},I_{2}} = (\lambda_{i_{1}}, \cdots, \lambda_{i_{l}}).
\end{equation}
Here we consider $\Lambda_{I_{1},I_{2}}$ as a gluing parameter for
$\mathcal{M}_{I_{2}}$. Define $\Lambda_{I_{1}}(I_{1}-I_{2}) \in [0,
+\infty)^{|I_{1}|}$ as
\begin{equation}
  \Lambda_{I_{1}}(I_{1} - I_{2}) (i) =
       \begin{cases}
          0 & r_{i} \in I_{2}, \\
          \lambda_{i} & r_{i} \notin I_{2}.
       \end{cases}
\end{equation}

For example, suppose $I_{1}= \{ r_{0}, r_{1}, r_{2}, r_{3}, r_{4}
\}$, $I_{2} = \{ r_{0}, r_{2}, r_{4} \}$ and $\Lambda_{I_{1}} = (5,
6, 7)$, then $\Lambda_{I_{1},I_{2}}=(6)$ and $\Lambda_{I_{1}}(I_{1}
- I_{2}) = (5,0,7)$.

Suppose $I_{1} = \{ r_{0}, \cdots, r_{k+1} \}$, $I_{2} = \{ r'_{0},
\cdots, r'_{l+1} \}$ and $r_{k+1} = r'_{0}$. Define
\begin{equation}
I_{1} \cdot I_{2} = \{ r_{0}, \cdots, r_{k+1}, r'_{1}, \cdots,
r'_{l+1} \}.
\end{equation}
If $x_{1} = (a_{1}, \cdots, a_{k+1}) \in \mathcal{M}_{I_{1}}$ and
$x_{2} = (a'_{1}, \cdots, a'_{l+1}) \in \mathcal{M}_{I_{2}}$, then
define
\begin{equation}
x_{1} \cdot x_{2} = (a_{1}, \cdots, a_{k+1}, a'_{1}, \cdots,
a'_{l+1}) \in \mathcal{M}_{I_{1}} \times \mathcal{M}_{I_{2}} =
\mathcal{M}_{I_{1} \cdot I_{2}}.
\end{equation}
Suppose $\Lambda_{I_{1}} = (\lambda_{1}, \cdots, \lambda_{|I_{1}|})$
and $\Lambda_{I_{2}} = (\lambda'_{1}, \cdots, \lambda'_{|I_{2}|})$,
define
\begin{equation}
\Lambda_{I_{1}} \cdot \Lambda_{I_{2}} = (\lambda_{1}, \cdots,
\lambda_{|I_{1}|}, 0, \lambda'_{1}, \cdots, \lambda'_{|I_{2}|}).
\end{equation}
In particular, if $|I_{1}|=0$, then $\Lambda_{I_{1}} \cdot
\Lambda_{I_{2}} = (0, \lambda'_{1}, \cdots, \lambda'_{|I_{2}|})$. If
$|I_{2}|=0$, then $\Lambda_{I_{1}} \cdot \Lambda_{I_{2}} =
(\lambda_{1}, \cdots, \lambda_{|I_{1}|}, 0)$. If $|I_{1}| = |I_{2}|
= 0$, then $\Lambda_{I_{1}} \cdot \Lambda_{I_{2}} = (0)$.

Suppose $I = \{ r_{0}, r_{1}, \cdots, r_{k+1} \}$ is a critical
sequence. Recall that an element $x \in \mathcal{M}_{I}$ is a
(un)broken trajectory which is broken at the points $r_{i}$ ($i=1,
\cdots, k$). A gluing should be a map $G_{I}: \mathcal{M}_{I} \times
[0, \epsilon_{I})^{|I|} \longrightarrow \overline{\mathcal{M}(r_{0},
r_{|I|+1})}$ for some $\epsilon_{I} > 0$. For all $(x, \Lambda_{I})
\in \mathcal{M}_{I} \times [0, \epsilon_{I})^{|I|}$, we have
$\Lambda_{I} = (\lambda_{1}, \cdots, \lambda_{|I|})$ is a parameter
of gluing, and $G_{I}(x, \Lambda_{I})$ is the (un)broken trajectory
glued from $x$. We expect that $G_{I}(x, \Lambda_{I})$ is not
broken at $r_{i}$ if and only if $\lambda_{i} > 0$. Thus we can
interpret the gluing map as a collaring map, which leads to the
following definition.

\begin{definition}\label{definition_gluing}
A map $G_{I}: \mathcal{M}_{I} \times [0, \epsilon_{I})^{|I|}
\rightarrow \overline{\mathcal{M}(r_{0}, r_{|I|+1})}$ for some
$\epsilon_{I} > 0$ is a gluing map if it satisfies the following
properties. (1). It is a smooth embedding. In particular, if
$|I|=0$, $G_{I}: \mathcal{M}_{I} = \mathcal{M}(r_{0}, r_{1})
\rightarrow \overline{\mathcal{M}(r_{0}, r_{1})}$ is the inclusion.
(2). It satisfies the stratum condition, i.e., suppose $I = \{
r_{0}, r_{1}, \cdots, r_{k+1} \}$, $\Lambda_{I} = (\lambda_{1},
\cdots, \lambda_{|I|}) \in [0, \epsilon_{I})^{|I|}$, $I_{1} \preceq
I$, and $\lambda_{i} = 0$ if and only if $r_{i} \in I_{1}$, then for
all $x \in \mathcal{M}_{I}$, we have $G_{I}(x, \Lambda_{I}) \in
\mathcal{M}_{I_{1}}$.
\end{definition}

Now we give two examples to illustrate the compatibility issue of
gluing.

Suppose the gluing maps are defined for all critical sequences.
Suppose $I_{1}= \{ r_{0}, r_{1}, r_{2}, r_{3}, r_{4} \}$, $I_{2} =
\{ r_{0}, r_{2}, r_{4} \}$, $\Lambda_{I_{1}} = (\lambda_{1},
\lambda_{2}, \lambda_{3})$, $\lambda_{1} > 0$, $\lambda_{3} > 0$,
and $x \in \mathcal{M}_{I_{1}}$. Gluing $x$ at the points $r_{1}$
and $r_{3}$ at first, we get $y = G_{I_{1}}(x, \lambda_{1}, 0,
\lambda_{3}) \in \mathcal{M}_{I_{2}}$. Do we have $G_{I_{2}}(y,
\lambda_{2}) = G_{I_{1}}(x, \lambda_{1}, \lambda_{2}, \lambda_{3})$?
This is a question about the compatibility for a fixed critical pair
$(r_{0}, r_{4})$.

Suppose $I_{1}= \{ r_{0}, r_{1}, r_{2} \}$, $I_{2} = \{ r_{2},
r_{3}, r_{4} \}$, $\Lambda_{I_{1}} = (\lambda_{1})$,
$\Lambda_{I_{2}} = (\lambda_{2})$, $x_{1} \in \mathcal{M}_{I_{1}}$
and $x_{2} \in \mathcal{M}_{I_{2}}$. Gluing $x_{1}$ and $x_{2}$, we
get $y_{1} = G_{I_{1}}(x_{1}, \lambda_{1}) \in \mathcal{M}(r_{0},
r_{2})$ and $y_{2} = G_{I_{2}}(x_{2}, \lambda_{2}) \in
\mathcal{M}(r_{2}, r_{4})$. Do we have $G_{I_{1} \cdot I_{2}} (x_{1}
\cdot x_{2}, \lambda_{1}, 0, \lambda_{2}) = (y_{1}, y_{2})$? This is
a question about the compatibility for different critical pairs.

The following theorem answers the above two questions.

\begin{theorem}\label{theorem_gluing}
Under Assumption \ref{assumption_M_I}, the gluing maps (see Definition
\ref{definition_gluing}) can be defined for all critical sequences.
They satisfy the following compatibility:

(1). (Compatibility for one critical Pair). Suppose $I_{2} \preceq
I_{1}$, let $\epsilon = \min \{ \epsilon_{I_{1}}, \epsilon_{I_{2}}
\}$. Then, for all $x \in \mathcal{M}_{I_{1}}$ and $\Lambda_{I_{1}}
= (\lambda_{1}, \cdots, \lambda_{|I_{1}|}) \in [0, \epsilon)^{|I_{1}|}$
such that $\lambda_{i}> 0$ when $r_{i}
\notin I_{2}$, we have
\begin{equation}
G_{I_{1}}(x, \Lambda_{I_{1}}) = G_{I_{2}} (G_{I_{1}}(x,
\Lambda_{I_{1}}(I_{1} - I_{2})), \Lambda_{I_{1},I_{2}}).
\end{equation}

(2). (Compatibility for Critical Pairs). Suppose $I_{1} = \{ r_{0},
\cdots, r_{k+1} \}$ and $I_{2} = \{ r_{k+1}, \cdots, r_{n} \}$. Let
$\epsilon = \min \{ \epsilon_{I_{1}}, \epsilon_{I_{2}},
\epsilon_{I_{1} \cdot I_{2}} \}$, then for all $x_{1} \in
\mathcal{M}_{I_{1}}$, $x_{2} \in \mathcal{M}_{I_{2}}$,
$\Lambda_{I_{1}} \in [0, \epsilon)^{|I_{1}|}$, and $\Lambda_{I_{2}}
\in [0, \epsilon)^{|I_{2}|}$, we have
\begin{eqnarray}
G_{I_{1} \cdot I_{2}}(x_{1} \cdot x_{2}, \Lambda_{I_{1}} \cdot
\Lambda_{I_{2}}) & = & (G_{I_{1}}(x_{1}, \Lambda_{I_{1}}),
G_{I_{2}}(x_{2}, \Lambda_{I_{2}})) \\
& \in & \overline{\mathcal{M}(r_{0}, r_{k+1})} \times
\overline{\mathcal{M}(r_{k+1}, r_{n})}. \nonumber
\end{eqnarray}
\end{theorem}

Theorem \ref{theorem_gluing} will follow from a more general Theorem
\ref{theorem_collar}.

We introduce a traditional notation of gluing as in the Introduction
(see e.g. \cite[p.\ 529]{floer1}). Suppose $\gamma_{1} \in
\mathcal{M}(p,r)$ and $\gamma_{2} \in \mathcal{M}(r,q)$ are two
trajectories. We denote the gluing map $G_{\{p,r,q\}}(\gamma_{1},
\gamma_{2}, \lambda)$ by $\gamma_{1} \#_{\lambda} \gamma_{2}$.
From Theorem \ref{theorem_gluing} we immediately derive the following.

\begin{theorem}\label{theorem_tradition_gluing}
Under Assumption \ref{assumption_M_I}, there exist
$\epsilon_{I} > 0$ for all critical sequences $I$ with $|I|=1$ or
$|I|=2$. For all $\{ r_{0}, r_{1}, r_{2} \}$, the gluing $\gamma_{1}
\#_{\lambda} \gamma_{2}$ can be defined for $(\gamma_{1},
\gamma_{2}) \in \mathcal{M}(r_{0},r_{1}) \times
\mathcal{M}(r_{1},r_{2})$ and $\lambda \in
[0,\epsilon_{\{r_{0},r_{1},r_{2}\}})$. The gluing satisfies the
following associativity:

For all $\gamma_{1} \in \mathcal{M}(p_{1},p_{2})$, $\gamma_{2} \in
\mathcal{M}(p_{1},p_{2})$, $\gamma_{3} \in
\mathcal{M}(p_{2},p_{3})$, and $\lambda_{1}$, $\lambda_{2} \in
(0,\epsilon)$, where $\epsilon = \min\{
\epsilon_{\{p_{0},p_{1},p_{2}\}},\epsilon_{\{p_{1},p_{2},p_{3}\}},\epsilon_{\{p_{0},p_{1},p_{2},p_{3}\}}
\}$, we have
\begin{equation}
\left( \gamma_{1} \#_{\lambda_{1}} \gamma_{2} \right)
\#_{\lambda_{2}} \gamma_{3} \,\, = \,\, \gamma_{1} \#_{\lambda_{1}} \left(
\gamma_{2} \#_{\lambda_{2}} \gamma_{3}  \right).
\end{equation}
\end{theorem}
\begin{proof}
\begin{eqnarray*}
& & \left( \gamma_{1} \#_{\lambda_{1}} \gamma_{2} \right)
\#_{\lambda_{2}} \gamma_{3} \\
& = & G_{\{p_{0}, p_{2}, p_{3}\}} \left( G_{\{p_{0}, p_{1}, p_{2}\}}
(\gamma_{1}, \gamma_{2}, \lambda_{1}),
\gamma_{3}, \lambda_{2} \right) \\
& = & G_{\{p_{0}, p_{2}, p_{3}\}} \left( G_{\{p_{0}, p_{1}, p_{2},
p_{3}\}}
(\gamma_{1}, \gamma_{2}, \gamma_{3}, \lambda_{1}, 0), \lambda_{2} \right) \\
& = & G_{\{p_{0}, p_{1}, p_{2}, p_{3}\}} (\gamma_{1}, \gamma_{2},
\gamma_{3}, \lambda_{1}, \lambda_{2}).
\end{eqnarray*}
Here we have used the (2) of Theorem \ref{theorem_gluing} in the
second equality and the (1) of Theorem \ref{theorem_gluing} in the
third equality.

Similarly,
\[
\gamma_{1} \#_{\lambda_{1}} \left( \gamma_{2} \#_{\lambda_{2}}
\gamma_{3}  \right) = G_{\{p_{0}, p_{1}, p_{2}, p_{3}\}}
(\gamma_{1}, \gamma_{2}, \gamma_{3}, \lambda_{1}, \lambda_{2}).
\]
This completes the proof.
\end{proof}

\begin{remark}
Suppose $I = \{ r_{0}, \cdots, r_{n+1} \}$ is a critical sequence.
Let $\epsilon = \min \{ \epsilon_{J} \mid  \text{ $J \subseteq I$,
and $|J|=1$ or $2$.} \}$. Then, for $(\gamma_{1}, \gamma_{2}) \in
\mathcal{M}(r_{i},r_{j}) \times \mathcal{M}(r_{j},r_{k})$, the
gluing $\gamma_{1} \#_{\lambda} \gamma_{2}$ in Theorem
\ref{theorem_tradition_gluing} can be defined for $\lambda \in
[0,\epsilon)$. And the gluing satisfies the associativity. Thus we
can define $G_{J}$ on $\mathcal{M}_{J} \times (0, \epsilon)^{|J|}$
for any $J \subseteq I$ by inductive gluing of pairs of
trajectories. The definition of $G_{J}$ does not depend on the order
of the pairwise gluing.
\end{remark}

By \cite[Theorem 3.3]{qin1} and \cite[Theorem 7.5]{qin2}, Assumption
\ref{assumption_M_I} holds in certain cases. Thus Theorems
\ref{theorem_gluing} and \ref{theorem_tradition_gluing} lead to the
following two propositions. See \cite[Definition
2.16]{qin1} for the definition of a locally trivial metric.

\begin{proposition}\label{proposition_infinite_gluing}
Suppose $M$ is a complete Hilbert-Riemannian manifold equipped with Morse function $f$
satisfying Condition (C) and having finite indices. Assume that the metric on $M$ is
locally trivial and $-\nabla f$ satisfies transversality. Give
$\mathcal{M}(p,q)$ the smooth structure induced from
$\mathcal{D}(p)$ and $\mathcal{A}(q)$. Then there exist smooth
structures on $\overline{\mathcal{M}(p,q)}$ and gluing maps which
satisfy the compatibility and associativity in Theorems
\ref{theorem_gluing} and \ref{theorem_tradition_gluing}.
\end{proposition}

\begin{proposition}\label{proposition_compact_gluing}
Suppose $M$ is a compact Riemannian manifold equipped with with Morse function $f$. Assume
$-\nabla f$ satisfies tranversality. Then
there exist smooth structures on $\mathcal{M}(p,q)$ and
$\overline{\mathcal{M}(p,q)}$ and  gluing maps which satisfy
the compatibility and associativity in Theorems \ref{theorem_gluing}
and \ref{theorem_tradition_gluing}.
\end{proposition}

\begin{remark}\label{remark_smooth_2}
Proposition \ref{proposition_compact_gluing} is based on
\cite{qin2}. In the case of a compact $M$, it has the advantage that
the metric is allowed to be general. However, the smooth structure
on $\mathcal{M}(p,q)$ may be different from the natural one when the
metric is not locally trivial.
\end{remark}

\section{Generalization}\label{section_generalization}
The proof of Theorem \ref{theorem_gluing} actually does not
directly depend on the speciality of Morse theory. Therefore, we will
generalize the results to Theorem \ref{theorem_collar} which is about
collaring maps of manifolds with faces.

\begin{definition}\label{definition_manifold_with_corner}
An $n$-dimensional {\it smooth manifold with corners} is a space defined
in the same way as a smooth manifold except that its atlases are
open subsets of $[0, + \infty)^{n}$.
\end{definition}

If $L$ is a smooth manifold with corners, $x \in L$, a neighborhood
of $x$ is diffeomorphic to $(0, \epsilon)^{n-k} \times [0,
\epsilon)^{k}$, then define $c(x) = k$. Clearly, $c(x)$ does not
depend on the choice of atlas.

\begin{definition}\label{definition_k_stratum}
Suppose $L$ is a smooth manifold. We call $\{ x \in L \mid c(x) = k
\}$ the $k$-stratum of $L$. Denote it by $\partial^{k} L$.
\end{definition}

Clearly, $\partial^{k} L$ is a submanifold \textit{without} corners
inside $L$, its codimension is $k$.

\begin{definition}\label{definition_manifold_with_face} (c.f.\ \cite{janich}).
A {\it smooth manifold $L$ with faces} is a smooth manifold with corners
such that each $x$ belongs to the closures of $c(x)$ different
components of $\partial^{1} L$.
\end{definition}

Now we introduce the notation $\Omega$, $``\succeq"$, $I$ and $\mathcal{M}(p,q)$
as in Section \ref{section_main_theorems}. However, in the present context
they are generalizations: they are \textit{independent} of Morse theory.

Suppose $\Omega$ is a partially ordered set with a partial order
$``\succeq"$. Suppose $I = \{ r_{0}, r_{1}, \cdots, r_{k+1} \}$ is a
finite chain of $\Omega$, i.e., $I \subseteq \Omega$ and $r_{i}
\succ r_{i+1}$. We call $r_{0}$ the head of $I$ and $r_{k+1}$ the
tail of $I$. Define the length of $I$ as $|I|=k$. If $J \subseteq
I$, $J = \{ r'_{0}, \cdots, r'_{l+1} \}$, $r'_{0}=r_{0}$ and
$r'_{l+1}=r_{k+1}$, i.e. $J = \{ r_{0}, r_{i_{1}}, \cdots,
r_{i_{l}}, r_{k+1} \}$, denote them by $J \preceq I$. Suppose $I_{1}
= \{ r_{0}, \cdots, r_{k+1} \}$ and $I_{2} = \{ r_{k+1}, \cdots,
r_{n} \}$ are two chains. Define $ I_{1} \cdot I_{2} = \{ r_{0},
\cdots, r_{n} \}$, which is also a chain.

Suppose a finite dimensional manifold $\mathcal{M}(p,q)$ is defined
for each pair $(p,q) \subseteq \Omega$ such that $p \succ q$. For
the above chain $I$, define $\mathcal{M}_{I} = \prod_{i=0}^{|I|}
\mathcal{M}(r_{i},r_{i+1})$.

\begin{assumption}\label{assumption_set}
The partially ordered set $\Omega$ is countable. The finite
dimensional manifolds $\mathcal{M}(p,q)$ can be compactified to be
$\overline{\mathcal{M}(p,q)}$ which are compact smooth manifolds
with faces. These $\overline{\mathcal{M}(p,q)}$ satisfy the
following conditions:

(1). We have $\overline{\mathcal{M}(p,q)} = \bigsqcup_{I}
\mathcal{M}_{I}$, where the disjoint is over all finite chains $I$
with head $p$ and tail $q$. The $k$-stratum of
$\overline{\mathcal{M}(p,q)}$ is $\bigsqcup_{|I|=k}
\mathcal{M}_{I}$, and each $\mathcal{M}_{I}$ is an open subset of
the $k$-stratum. The smooth structure of
$\overline{\mathcal{M}(p,q)}$ is compatible with those of
$\mathcal{M}_{I}$.

(3). Suppose $p \succ r \succ q$, then the natural inclusion
$\overline{\mathcal{M}(p,r)} \times \overline{\mathcal{M}(r,q)}
\hookrightarrow \overline{\mathcal{M}(p,q)}$ is a smooth embedding.
\end{assumption}

We introduce the following definitions similar to Section
\ref{section_main_theorems}. Use $\Lambda_{I} = (\lambda_{1},
\cdots, \lambda_{|I|})$ to represent the collaring parameter for
$\mathcal{M}_{I}$. Define $\Lambda_{I_{1}} (I_{1} - I_{2})$,
$\Lambda_{I_{1},I_{2}}$ and $\Lambda_{I_{1}} \cdot \Lambda_{I_{2}}$.
Also for $x_{1} \in \mathcal{M}_{I_{1}}$ and $x_{2} \in
\mathcal{M}_{I_{2}}$, define $x_{1} \cdot x_{2} \in
\mathcal{M}_{I_{1} \cdot I_{2}}$.

Define the collaring map $G_{I}: \mathcal{M}_{I} \times
[0,\epsilon_{I})^{|I|} \rightarrow
\overline{\mathcal{M}(r_{0},r_{|I|+1})}$ as Definition
\ref{definition_gluing}.

The proof of the following theorem is given in Section
\ref{section_proof}.

\begin{theorem}\label{theorem_collar}
Under Assumption \ref{assumption_set}, the collaring maps $G_{I}$
can be defined for all finite chains $I$ of $\Omega$. These maps
satisfy the following compatibility:

(1). Suppose $I_{2} \preceq I_{1}$, let $\epsilon = \min \{
\epsilon_{I_{1}}, \epsilon_{I_{2}} \}$. Then, for all $x \in
\mathcal{M}_{I_{1}}$ and $\Lambda_{I_{1}} \in [0,
\epsilon)^{|I_{1}|}$ such that $\lambda_{i}> 0$ when $r_{i} \notin
I_{2}$, we have
\begin{equation}\label{theorem_collar_1}
G_{I_{1}}(x, \Lambda_{I_{1}}) = G_{I_{2}} (G_{I_{1}}(x,
\Lambda_{I_{1}}(I_{1} - I_{2})), \Lambda_{I_{1},I_{2}}).
\end{equation}

(2). Suppose $I_{1} = \{ r_{0}, \cdots, r_{k+1} \}$ and $I_{2} = \{
r_{k+1}, \cdots, r_{n} \}$. Let $\epsilon = \min \{
\epsilon_{I_{1}}, \epsilon_{I_{2}},$ $\epsilon_{I_{1} \cdot I_{2}}
\}$, then for all $x_{1} \in \mathcal{M}_{I_{1}}$, $x_{2} \in
\mathcal{M}_{I_{2}}$, $\Lambda_{I_{1}} \in [0, \epsilon)^{|I_{1}|}$,
and $\Lambda_{I_{2}} \in [0, \epsilon)^{|I_{2}|}$, we have
\begin{equation}\label{theorem_collar_2}
G_{I_{1} \cdot I_{2}}(x_{1} \cdot x_{2}, \Lambda_{I_{1}} \cdot
\Lambda_{I_{2}}) =  (G_{I_{1}}(x_{1}, \Lambda_{I_{1}}),
G_{I_{2}}(x_{2}, \Lambda_{I_{2}})).
\end{equation}
\end{theorem}

\section{Face Structures}\label{section_face_structures}
In order to prove Theorem \ref{theorem_collar}, we first study the
face structures.

Suppose $L$ is manifold with faces. The closure of a component of
$\partial^{1} L$ (see Definition \ref{definition_k_stratum}) is
still connected. Following the terminology of \cite{janich}, we have
the following definition.

\begin{definition}\label{definition_face}
We call the closure of a component of $\partial^{1} L$ a connected
(closed) face of $L$. We call any union of pairwise disjoint
connected faces a face of $L$.
\end{definition}

Thus, if $F$ is a face of $L$, then $F = \bigsqcup_{\alpha \in
\mathfrak{A}} C_{\alpha}$, where $C_{\alpha}$ is the closure of
$C^{\circ}_{\alpha}$ and $C^{\circ}_{\alpha}$ is a component of
$\partial^{1} L$. As pointed in \cite{janich}, $F$ is still a
manifold with corners. We have the following result which is trivial
when $\mathfrak{A}$ is a finite set.

\begin{lemma}\label{lemma_face_1}
Using the notation as the above, we have that $F$ is a smoothly
embedded submanifold with corners inside $L$. The components of $F$
are $C_{\alpha}$. The interior of $F$ (i.e. $\partial^{0} F$) is
$\bigsqcup_{\alpha \in \mathfrak{A}} C^{\circ}_{\alpha}$ and $F$ is
a closed subset of $L$.
\end{lemma}
\begin{proof}
First, we show that $C_{\alpha}$ is a submanifold with corners and
its $0$-stratum is $C^{\circ}_{\alpha}$. It suffices to show that,
for each $x \in C_{\alpha}$, there exists an open neighborhood
$U_{x}$ of $x$ such that $U_{x} \cap C_{\alpha}$ has the desired
corner structure.

We can choose $U_{x}$ such that it has the chart $(-\epsilon,
\epsilon)^{n-l} \times [0, \epsilon)^{l}$ and $x$ has the coordinate
$(0, \cdots, 0)$. Clearly,
\[
U_{x} \cap C^{\circ}_{\alpha} \subseteq \bigsqcup_{i=1}^{l} \left[
(-\epsilon, \epsilon)^{n-l} \times (0, \epsilon)^{i-1} \times \{0\}
\times (0, \epsilon)^{l-i} \right],
\]
and $U_{x} \cap C^{\circ}_{\alpha} \neq \emptyset$. We may assume
$[(-\epsilon, \epsilon)^{n-l} \times \{0\} \times (0,
\epsilon)^{l-1}] \cap C^{\circ}_{\alpha} \neq \emptyset$. Since
$(-\epsilon, \epsilon)^{n-l} \times \{0\} \times (0,
\epsilon)^{l-1}$ is connected and contained in $\partial^{1} L$, and
$C^{\circ}_{\alpha}$ is a component of $\partial^{1} L$, we infer
that $(-\epsilon, \epsilon)^{n-l} \times \{0\} \times (0,
\epsilon)^{l-1} \subseteq C^{\circ}_{\alpha}$. By Definition
\ref{definition_manifold_with_face}, it's easy to see $U_{x} \cap
C^{\circ}_{\alpha} = (-\epsilon, \epsilon)^{n-l} \times \{0\} \times
(0, \epsilon)^{l-1}$. Since, $U_{x}$ is open, we have $U_{x} \cap
C_{\alpha}$ is the relative closure of $U_{x} \cap
C^{\circ}_{\alpha}$ in $U_{x}$. In other words, $U_{x} \cap
C_{\alpha} = (-\epsilon, \epsilon)^{n-l} \times \{0\} \times [0,
\epsilon)^{l-1}$ and the $0$-stratum of $U_{x} \cap C_{\alpha}$ is
contained in $C^{\circ}_{\alpha}$. Thus we get the desired corner
structure.

Second, we show that $F$ is a manifold with corners.

Since $C_{\alpha}$ has no intersection with other $C_{\beta}$, by
the above argument, we can see that the above open neighborhood
$U_{x}$ has no intersection with other $C_{\beta}$. Thus $\bigcup_{x
\in C_{\alpha}} U_{x}$ is an open neighborhood of $C_{\alpha}$ which
has no intersection with other $C_{\beta}$. So $C_{\alpha}$ is
relatively open in $F$. This verifies the manifold structure of $F$.

Finally, we show that $F$ is a closed subset of $L$. Suppose $x$ is
in the closure of $F$, then $x$ can be approximated by points in $F$
and thus by points in $\bigsqcup_{\alpha \in \mathfrak{A}}
C^{\circ}_{\alpha}$. By the above argument, it's easy to see that
$x$ belongs to some $C_{\alpha}$.
\end{proof}

\begin{lemma}\label{lemma_face_2}
Suppose $L$ is an $n$ dimensional manifold with faces. Suppose
$F_{i}$ ($i=1, \cdots k$) are faces of $L$ such that their interiors
are pairwise disjoint and $\bigcap_{i=1}^{k} F_{i}$ is nonempty.
Then $\bigcap_{i=1}^{k} F_{i}$ is an $n-k$ dimensional smoothly
embedded submanifold with corners inside $L$.
\end{lemma}
\begin{proof}
Let $x$ be an arbitrary point in $\bigcap_{i=1}^{k} F_{i}$. It
suffices to prove that there exists an open neighborhood $U$ of $x$
such that $U \cap \bigcap_{i=1}^{k} F_{i}$ has a corner structure.

For each $i$, $x$ belongs to an unique component of $F_{i}$. Since
this component is relatively open in $F_{i}$, we can choose $U$
small enough such that $U$ has no intersection with other
components. Thus we may assume $F_{i}$ is connected.

By the proof of Lemma \ref{lemma_face_1}, we can choose $U$ such
that it has a chart $(-\epsilon, \epsilon)^{n-l} \times [0,
\epsilon)^{l}$, $x$ has the coordinate $(0, \cdots, 0)$ and $U \cap
F_{1} = (-\epsilon, \epsilon)^{n-l} \times \{0\} \times [0,
\epsilon)^{l-1}$. Since the interior of $F_{i}$ are pairwise
disjoint, repeating this argument, we get $U \cap F_{i} =
(-\epsilon, \epsilon)^{n-l} \times [0, \epsilon)^{i-1} \times \{0\}
\times [0, \epsilon)^{l-i}$. Thus $U \cap \bigcap_{i=1}^{k} F_{i} =
(-\epsilon, \epsilon)^{n-l} \times \{0\}^{k} \times [0,
\epsilon)^{l-k}$. This verifies the corner structure.
\end{proof}

We introduce some other concepts following \cite{douady}.

\begin{definition}\label{definition_tangent_sector}
Suppose $L$ is a manifold with corners. For all $x \in L$,
\[
A_{x}L = \{ v \in T_{x}L \mid \text{$v = \gamma'(0)$ for some smooth
curve $\gamma:$ $[0,\epsilon) \longrightarrow L$.} \}
\]
is the tangent sector of $L$ at $x$.
\end{definition}

Definition \ref{definition_tangent_sector} is equivalent to the
\textit{secteur tangent} in \cite[p.\ 3]{douady}.

\begin{definition}\label{definition_normal_sector}
Suppose $L_{1}$ is a submanifold \textit{without} corners inside $L$
and $x \in L_{1}$, we define the normal sector $A_{x}(L_{1}, L) =
A_{x}L / T_{x} L_{1}$.
\end{definition}

In \cite{douady}, $A_{x}(L_{1}, L)$ is called \textit{secteur
transverse}.

Define the tangent sector bundle $AL$ as the subbundle of $TL$ with
fibers $A_{x}L$. Define the normal bundle $N(L_{1},L)$ as the bundle
whose fibers are the normal space $N_{x}(L_{1},L) = T_{x} L / T_{x}
L_{1}$. Define the normal sector bundle $A(L_{1},L)$ as the
subbundle of $N(L_{1},L)$ with fiber $A_{x}(L_{1},L)$ and
$A_{L_{1}}L$ as the restriction of $AL$ to $L_{1}$.

\begin{lemma}\label{lemma_face_3}
Under the assumption of Lemma \ref{lemma_face_2}, assume that
$L_{1}$ is an open subset of $\partial^{k} L$ and $L_{1} \subseteq
\bigcap_{i=1}^{k} F_{i}$. Then there exist smooth sections $e_{i}$
of $A_{L_{1}}L$ ($i=1, \cdots, k$) satisfying the following stratum
condition: (1). $e_{i} \in A_{L_{1}} (\bigcap_{j \neq i} F_{j})$;
(2). $\{ \pi e_{1}, \cdots, \pi e_{k} \}$ is linearly independent
everywhere and all elements in $A_{x}(L_{1},L)$ can be linearly
represented by $\{ \pi e_{1}(x), \cdots, \pi e_{k}(x) \}$ with
nonnegative coefficients, where $\pi: A_{L_{1}}L \rightarrow
A(L_{1},L)$ is the natural projection.
\end{lemma}
\begin{proof}
Suppose $x \in L_{1}$, by the proof of Lemma \ref{lemma_face_2},
there exists a neighborhood $U$ of $x$ such that $U$ has a chart
$(-\epsilon, \epsilon)^{n-k} \times [0, \epsilon)^{k}$, $x$ has the
coordinate $(0, \cdots, 0)$, $U \cap L_{1} = (-\epsilon,
\epsilon)^{n-k} \times \{0\}^{k}$ and $U \cap F_{i} = (-\epsilon,
\epsilon)^{n-k} \times [0, \epsilon)^{i-1} \times \{0\} \times [0,
\epsilon)^{k-i}$. Thus $U \cap \bigcap_{j \neq i} F_{j} =
(-\epsilon, \epsilon)^{n-k} \times \{0\}^{i-1} \times [0, \epsilon)
\times \{0\}^{k-i}$. Obviously, for any vector $e_{i}(x) \in A_{x}
(\bigcap_{j \neq i} F_{j}) - T_{x}L_{1}$, we have $\{ \pi e_{1}(x),
\cdots, \pi e_{k}(x) \}$ satisfies the desired property in
$A_{x}(L_{1}, L)$.

Since $L_{1}$ is an open subset of the $1$-stratum of $\bigcap_{j
\neq i} F_{j}$, we can choose a smooth inward normal section $e_{i}$
along $L_{1}$.
\end{proof}

In the case of Assumption \ref{assumption_set}, it's easy to see
that $\overline{\mathcal{M}(p,q)}$ is a manifold with faces
$\overline{\mathcal{M}(p,r)} \times \overline{\mathcal{M}(r,q)}$.
The interiors of these faces are $\mathcal{M}(p,r) \times
\mathcal{M}(r,q)$ which are pairwise disjoint. Suppose $I = \{ p,
r_{1}, \cdots, r_{k}, q \}$ is a chain of $\Omega$. Let $I_{i} = \{
p, r_{1}, \cdots, r_{i-1}, r_{i+1}, \cdots, r_{k}, q \}$. Then
$\mathcal{M}_{I}$ is the interior of $\bigcap_{i=1}^{k}
\overline{\mathcal{M}(p,r_{i})} \times
\overline{\mathcal{M}(r_{i},q)}$, and $\bigcap_{j \neq i}
\overline{\mathcal{M}(p,r_{j})} \times
\overline{\mathcal{M}(r_{j},q)} = \overline{\mathcal{M}_{I_{i}}}$.
By Lemma \ref{lemma_face_3}, we have the following corollary.

\begin{corollary}\label{corollary_normal_frame}
There exists a smooth frame $\{e_{1}, \cdots, e_{k}\}$ along
$\mathcal{M}_{I}$ satisfying the following stratum condition: (1).
$e_{i} \in A_{\mathcal{M}_{I}} \overline{\mathcal{M}_{I_{i}}}$; (2).
$\{ \pi e_{1}, \cdots, \pi e_{k} \}$ is linearly independent
everywhere and all elements in $A_{x}(\mathcal{M}_{I},
\overline{\mathcal{M}(p,q)})$ can be linearly represented by $\{ \pi
e_{1}(x), \cdots, \pi e_{k}(x) \}$ with nonnegative coefficients,
where $\pi: A_{\mathcal{M}_{I}} \overline{\mathcal{M}(p,q)}
\rightarrow A(\mathcal{M}_{I}, \overline{\mathcal{M}(p,q)})$ is the
natural projection.
\end{corollary}

For a manifold $L$ with corners, \cite[p.\ 8]{douady} shows that there
exists a connection on $L$ such that all strata are totally
geodesic. (See \cite[Chapter 4]{chern_chen_lam} for a detailed
treatment of connections.) Suppose $L_{1}$ is a stratum of $L$. Then
by the above connection and the exponential map, \cite{douady} shows
that an open neighborhood of $L_{1}$ in $A(L_{1}, L)$ is
diffeomorphic to an open neighborhood of $L_{1}$ in $L$. Thus by the
frame in Corollary \ref{corollary_normal_frame}, we get the
following lemma.

\begin{lemma}\label{lemma_embed_normal}
There is a smooth embedding $\varphi_{I}: \mathcal{M}_{I} \times
[0,1)^{|I|} \longrightarrow \overline{\mathcal{M}(p,q)}$ satisfying
the stratum condition (See (2) in Definition
\ref{definition_gluing}).
\end{lemma}

In order to prove Theorem \ref{theorem_collar}, we need some
connections even better than the above one. This leads to the
definition of the product connection. There are several ways to
define a connection on a manifold $L$. One is as follows. A
connection is to assign each smooth curve $\gamma: [0,1]
\longrightarrow L$ a parallel transport (or displacement)
$P_{\gamma}: T_{\gamma(0)} L \longrightarrow T_{\gamma(1)} L$ which
is a linear isomorphism. Suppose $L_{1}$ and $L_{2}$ are two
manifolds with corners. Clearly, $T(L_{1} \times L_{2}) = T L_{1}
\times T L_{2}$. We define the product connection on $L_{1} \times
L_{2}$ as follows.
\begin{definition}\label{definition_product_connection}
Let $\gamma = (\gamma_{1}, \gamma_{2}): [0,1] \longrightarrow L_{1}
\times L_{2}$ be a smooth curve. Define the parallel transport
$P_{\gamma}: T_{\gamma(0)} (L_{1} \times L_{2}) \longrightarrow
T_{\gamma(1)} (L_{1} \times L_{2})$ as $P_{\gamma}(v_{1}, v_{2}) =
(P_{\gamma_{1}} v_{1}, P_{\gamma_{2}} v_{2})$, where
$P_{\gamma_{i}}$ is the parallel transport along $\gamma_{i}$.
The connection assigning $P_{\gamma}$ is the product connection.
\end{definition}
For a product connection, a curve $\gamma$ in $L_{1} \times L_{2}$
is a geodesic if and only if both $\gamma_{1}$ and $\gamma_{2}$ are
geodesics. By Lemma \ref{lemma_embed_normal}, $\varphi_{I}$ pulls
back the connection on $\overline{\mathcal{M}(p,q)}$ to
$\mathcal{M}_{I} \times [0,1)^{|I|}$. Let $\gamma$ be a curve in
$\mathcal{M}_{I} \times [0,1)^{|I|}$ such that $\gamma(t)= (x,
\sigma(t))$, where $x \in \mathcal{M}_{I}$ and $\sigma$ is a
straight line in $[0,1)^{|I|}$. If $\sigma$ passes through the
origin, then $\gamma$ is a geodesic because $\varphi_{I}$ is defined
by the exponential map. Since $\mathcal{M}_{I}$ is totally geodesic
in $\overline{\mathcal{M}(p,q)}$, we infer that $\mathcal{M}_{I}$
has a connection. Moreover, $[0,1)^{|I|}$ also has its standard flat
connection. We can define the product connection of $\mathcal{M}_{I}
\times [0,1)^{|I|}$. The product connection coincides with the old
one on $T (\mathcal{M}_{I} \times \{0\}^{|I|})$, and $\varphi_{I}$
is still given by the exponential map under the new connection. This
new connection has its advantage over the old one. In particular,
for \textit{every} straight line $\sigma$ in $[0,1)^{|I|}$,
\textit{not} necessarily passing through the origin, $\gamma(t)= (x,
\sigma(t))$ is a geodesic of the new connection. This is important
in the proof of Theorem \ref{theorem_collar}.

\section{Proof of Theorem \ref{theorem_collar}}\label{section_proof}
Before proving Theorem \ref{theorem_collar}, we shall introduce some
definitions and notation.

\begin{definition}\label{definition_pair_length}
Suppose $(p,q) \subseteq \Omega$, where $\Omega$ is the set defined
in Assumption \ref{assumption_set}. If $p \nsucc q$, then define the
length of $(p,q)$ as $|p,q| = -1$. Otherwise, define the length of
$(p,q)$ as $|p,q| = \sup \{ |I| \mid \text{$I$ is a chain with head
$p$ and tail $q$} \}$.
\end{definition}

By (1) of Assumption \ref{assumption_set}, we know that $|p,q| \leq
\text{dim}(\overline{\mathcal{M}(p,q)}) < +\infty$.

By the compactness of $\overline{\mathcal{M}(p,q)}$ and (1) of
Assumption \ref{assumption_set}, there are only finitely many chains
$I$ with head $p$ and tail $q$.

Suppose $I_{1} = \{ r_{0}, \cdots, r_{k+1} \}$ and  $I_{2} = \{
r_{0}, r_{i_{1}}, \cdots, r_{i_{l}}, r_{k+1} \}$ are two chains of
$\Omega$ such that $I_{2} \preceq I_{1}$. Like Section
\ref{section_main_theorems}, if $\Lambda_{I_{2}} = (\lambda_{i_{1}},
\cdots, \lambda_{i_{l}}) \in [0, +\infty)^{|I_{2}|}$ is a collaring
parameter for for $\mathcal{M}_{I_{2}}$, then define
$\Lambda_{I_{2}, I_{1}} \in [0, +\infty)^{|I_{1}|}$ as
\begin{equation}\label{extension_parameter}
  \Lambda_{I_{2}, I_{1}} (i) =
       \begin{cases}
          \lambda_{i} & r_{i} \in I_{2}, \\
          0 & r_{i} \notin I_{2}.
       \end{cases}
\end{equation}
Here we consider $\Lambda_{I_{2},I_{1}}$ as a collaring parameter
for $\mathcal{M}_{I_{1}}$.

If $I_{i} \prec I$ ($i=1, \cdots, n$), then define
\begin{equation}
\Lambda_{I} + \Lambda_{I_{1}} + \cdots \Lambda_{I_{n}} = \Lambda_{I}
+ \Lambda_{I_{1},I} + \cdots \Lambda_{I_{n},I}.
\end{equation}
Clearly,
\[
\Lambda_{I_{1}} = \Lambda_{I_{1}}(I_{1} - I_{2}) + \Lambda_{I_{1},
I_{2}}
\]

For example, suppose $I_{1}= \{ r_{0}, r_{1}, r_{2}, r_{3}, r_{4}
\}$, $I_{2} = \{ r_{0}, r_{2}, r_{4} \}$ and $\Lambda_{I_{1}} = (5,
6, 7)$, then $\Lambda_{I_{1},I_{2}}=(6)$, and
\[
\Lambda_{I_{1}}(I_{1}-I_{2}) + \Lambda_{I_{1},I_{2}} = (5, 0, 7) +
(6) = (5, 0, 7) + (0, 6, 0) = (5, 6, 7) = \Lambda_{I_{1}}.
\]
If $\Lambda_{I_{2}} = (8)$, then $\Lambda_{I_{2}, I_{1}} = (0,8,0)$
and
\[
\Lambda_{I_{1}} + \Lambda_{I_{2}} = (5, 6, 7) + (8) = (5, 6, 7) +
(0,8,0) = (5, 14, 7).
\]

\begin{proof}[Proof of Theorem \ref{theorem_collar}.]
We shall define $G_{I}$ by exponential maps. This requires two
things. First, we need a frame satisfying the stratum condition (See
Corollary \ref{corollary_normal_frame}) in $A(\mathcal{M}_{I},
\overline{\mathcal{M}(p,q)})$. Second, we need a connection on
$\overline{\mathcal{M}(p,q)}$. The proof is to construct the above
two things by a double induction. The outer induction is on the
length $|p,q|$. We construct the desired $G_{I}$ in the case of
$|p,q|=n$ based on the hypothesis that all $G_{I}$ have been
constructed and satisfy (\ref{theorem_collar_1}) and
(\ref{theorem_collar_2}) for all $|p,q|<n$. The inner induction is
the process to construct $G_{I}$ for a fixed pair $(p,q)$.

\textit{(1). The first step of the outer induction (the induction on
$|p,q|$).}

When $|p,q|=0$, then $\mathcal{M}_{I} =
\overline{\mathcal{M}(p,q)}$, define $G_{I}: \mathcal{M}_{I}
\rightarrow \overline{\mathcal{M}(p,q)}$ as the identity.

\textit{(2). The second step of the outer induction (the induction
on $|p,q|$).}

Suppose we have constructed the desired $G_{I}$ for all pair $(p,q)$
such that $|p,q|<n$. We shall construct $G_{I}$ in the case of
$|p,q|=n$. The construction is the inner induction. Let $X_{k}$ be
the union of all $l$-strata of $\overline{\mathcal{M}(p,q)}$ with $l
\geq k$. Clearly, $X_{k+1} \subseteq X_{k}$, $X_{1}$ is the full
boundary of $\overline{\mathcal{M}(p,q)}$. We shall construct a
family of open sets $U_{k}$ such that $U_{k+1} \subseteq U_{k}$ and
$X_{k} \subseteq U_{k}$ by an downward induction on $k$. In other
words, we construct $U_{k}$ after having constructed $U_{k+1}$. For
each $k$, we shall construct $G_{I}: (\mathcal{M}_{I} \cap U_{k})
\times [0,\epsilon)^{|I|} \rightarrow \overline{\mathcal{M}(p,q)}$
such that $\text{Im} G_{I} \subseteq U_{k}$, and all $G_{I}$ satisfy
(\ref{theorem_collar_1}) and (\ref{theorem_collar_2}). We call such
a map $G_{I}$ in $U_{k}$, denote it by $G_{I}|_{U_{k}}$. Extend
$G_{I}$ with the step of the inner induction. Clearly, $U_{1}$
contains all $\mathcal{M}_{I}$ such that $|I| > 0$. If the
construction of $G_{I}|_{U_{1}}$ is finished, we shall complete the
proof by defining $G_{\{p,q\}}$ as the inclusion.

Since $|p,q|=n$, the stratum with the lowest dimension is the
$n$-stratum.

\textit{(I). The first step of the inner induction (the induction on
$U_{k}$).}

We shall construct $U_{n}$, $G_{I}|_{U_{n}}$, frames for
$\mathcal{M}_{I} \cap U_{n}$ and a connection providing all $G_{I}$
via the exponential map. Moreover, $(\mathcal{M}_{I} \cap U_{n})
\times [0,\epsilon)^{|I|}$ will also have a product connection (see
Definition \ref{definition_product_connection} and the comment
following it) if we pull back the connection on $U_{n}$ via $G_{I}$.

We know that $X_{n} = \bigcup_{|J|=n} \mathcal{M}_{J}$. By Lemma
\ref{lemma_embed_normal}, we can construct a smooth embedding
$\varphi_{J}: \mathcal{M}_{J} \times [0, \epsilon_{0})^{|J|}
\rightarrow \overline{\mathcal{M}(p,q)}$ satisfying the stratum
condition (See (2) in Definition \ref{definition_gluing}).
Furthermore, $\mathcal{M}_{J}$ is compact because it is closed (also
open) in the lowest dimensional stratum. Choose $\epsilon_{0}$ small
enough so that $\text{Im} \varphi_{J}$ are pairwise disjoint for all $J$ such
that $|J|=n$. Fix $J = \{p, r_{1}, \cdots, r_{n}, q \}$. Suppose
$J_{l} = \{p, r_{1}, \cdots, r_{l}\}$ and $J'_{l} = \{r_{l}, \cdots,
r_{n}, q\}$. Clearly, $|p, r_{l}| < n$ and $|r_{l}, q| < n$. By the
outer induction on $|p,q|$, $G_{J_{l}}$ and $G_{J'_{l}}$ have been
defined.

\begin{lemma}\label{lemma_collar_1}
There exists $\epsilon>0$. And $\varphi_{J}$ can be modified to be
defined on $\mathcal{M}_{J} \times [0, \epsilon)^{|J|}$ such that
for all $l \in \{ 1, \cdots, n \}$, we have
\[
\varphi_{J}(x_{1} \cdot x_{2}, \Lambda_{J_{l}} \cdot
\Lambda_{J'_{l}}) = (G_{J_{l}}(x_{1}, \Lambda_{J_{l}}),
G_{J'_{l}}(x_{2}, \Lambda_{J'_{l}})).
\]
\end{lemma}
\begin{proof}
For small $\epsilon$, $G_{J_{l}} \times G_{J'_{l}} (\mathcal{M}_{J}
\times \prod_{i=1, i \neq l}^{|J|} [0, \epsilon)) \subseteq \text{Im}
\varphi_{J}$, where $G_{J_{l}} \times G_{J'_{l}}(x_{1} \cdot x_{2},
\Lambda_{J_{l}}, \Lambda_{J'_{l}}) = (G_{J_{l}}(x_{1},
\Lambda_{J_{l}}), G_{J'_{l}}(x_{2}, \Lambda_{J'_{l}}))$.

Consider the following map $\phi_{l} = \varphi_{J}^{-1} \circ
(G_{J_{l}} \times G_{J'_{l}})$,
\[
\phi_{l}: \mathcal{M}_{J} \times \prod_{i=1, i \neq l}^{|J|} [0,
\epsilon) \rightarrow \text{\rm Im} \varphi_{J} \rightarrow \mathcal{M}_{J}
\times [0, \epsilon_{0})^{|J|}.
\]
We only need to prove that $\varphi_{J}$ can be modified such that
for all $l$,
\begin{equation}\label{lemma_collar_1_1}
\phi_{l}(x, \lambda_{1}, \cdots, \lambda_{l-1}, \lambda_{l+1},
\cdots, \lambda_{n}) = (x, \lambda_{1}, \cdots, \lambda_{l-1}, 0,
\lambda_{l+1}, \cdots, \lambda_{n}).
\end{equation}
Denote $(\lambda_{1}, \cdots, \lambda_{n})$ by $\Lambda_{J}$,
$(\lambda_{1}, \cdots, \lambda_{l-1}, \lambda_{l+1}, \cdots,
\lambda_{n})$ by $\Lambda_{J-l}$, $(\lambda_{1}, \cdots,
\lambda_{l-1})$ by $\Lambda_{J_{l}}$, and $(\lambda_{l+1}, \cdots,
\lambda_{n})$ by $\Lambda_{J'_{l}}$. Since $\text{Im}(G_{J_{l}} \times
G_{J'_{l}}) \subseteq \overline{\mathcal{M}(p,r_{l})} \times
\overline{\mathcal{M}(r_{l},q)}$ and $\varphi_{J}$ satisfies the
stratum condition, we have
\[
\phi_{l} (x, \Lambda_{J-l}) = (a, c_{1}, \cdots, c_{l-1}, 0,
c_{l+1}, \cdots, c_{n})
\]
where $a$ and $c_{i}$ are smooth functions of $x$ and
$\Lambda_{J-l}$.

Define $\theta_{l}: \mathcal{M}_{J} \times [0, \epsilon)^{|J|}
\rightarrow \mathcal{M}_{J} \times [0, \epsilon_{0})^{|J|}$ as
\begin{equation}\label{lemma_collar_1_2}
\theta_{l}(x, \Lambda_{J}) = (a, \cdots, c_{l-1}, \lambda_{l},
c_{l+1}, \cdots, c_{n}).
\end{equation}
Since $\phi_{l}$ is a smooth embedding, so is $\theta_{l}$. Since
$\mathcal{M}_{J}$ is compact, shrink $\epsilon_{0}$ if necessary, we
may assume $\theta_{l}^{-1}$ can be defined on $\mathcal{M}_{J}
\times [0, \epsilon_{0})^{|J|}$. Thus
\begin{eqnarray*}
& & (\varphi_{J} \circ \theta_{l})^{-1} \circ (G_{J_{l}} \times
G_{J'_{l}}) (x, \Lambda_{J-l}) \\
& = & \theta_{l}^{-1} \circ \phi_{l} (x,
\Lambda_{J-l}) \\
& = & (x, \lambda_{1}, \cdots, \lambda_{l-1}, 0, \lambda_{l+1},
\cdots, \lambda_{n}) \\
& = & (x, \Lambda_{J_{l}} \cdot \Lambda_{J'_{l}}).
\end{eqnarray*}
Modify $\varphi_{J}$ to be $\varphi_{J} \circ \theta_{l}$, we get
(\ref{lemma_collar_1_1}) is true for a fixed $l$ and some $\epsilon
> 0$.

In general, suppose we have proved (\ref{lemma_collar_1_1}) is true
for $l \in \{1, \cdots, j-1\}$, we shall modify $\varphi_{J}$ such
that (\ref{lemma_collar_1_1}) is true for all $l \in \{1, \cdots,
j\}$. Let $x = x_{1} \cdot x_{2} \cdot x_{3}$, where $x_{1} =
(a_{0}, \cdots, a_{l-1})$, $x_{2} = (a_{l}, \cdots, a_{j-1})$ and
$x_{3} = (a_{j}, \cdots, a_{n})$. Denote $\{ r_{l}, \cdots, r_{j}
\}$ by $J_{(l,j)}$ and $( \lambda_{l+1}, \cdots, \lambda_{j-1} )$ by
$\Lambda_{J_{(l,j)}}$.
\[
\phi_{j}(x, \Lambda_{J_{l}} \cdot \Lambda_{J_{(l,j)}},
\Lambda_{J'_{j}}) = \varphi_{J}^{-1} (G_{J_{j}}(x_{1} \cdot x_{2},
\Lambda_{J_{l}} \cdot \Lambda_{J_{(l,j)}}), G_{J'_{j}}(x_{3},
\Lambda_{J'_{j}})).
\]
Since $|p,r_{l}|<n$, by the outer inductive hypothesis, $G_{J_{j}}$
satisfies (\ref{theorem_collar_2}). Shrink $\epsilon$ if necessary,
we have
\[
G_{J_{j}}(x_{1} \cdot x_{2}, \Lambda_{J_{l}} \cdot
\Lambda_{J_{(l,j)}}) = (G_{J_{l}}(x_{1}, \Lambda_{J_{l}}),
G_{J_{(l,j)}}(x_{2}, \Lambda_{J_{(l,j)}})),
\]
Similarly,
\[
G_{J'_{l}}(x_{2} \cdot x_{3}, \Lambda_{J_{(l,j)}} \cdot
\Lambda_{J'_{j}}) = (G_{J_{(l,j)}}(x_{2}, \Lambda_{J_{(l,j)}}),
G_{J'_{j}}(x_{3}, \Lambda_{J'_{j}})).
\]
Thus
\[
G_{J_{j}} \times G_{J'_{j}} (x, \Lambda_{J_{l}} \cdot
\Lambda_{J_{(l,j)}}, \Lambda_{J'_{j}}) =  G_{J_{l}} \times
G_{J'_{l}} (x, \Lambda_{J_{l}}, \Lambda_{J_{(l,j)}} \cdot
\Lambda_{J'_{j}}).
\]
Then
\begin{eqnarray*}
& & \phi_{j}(x, \Lambda_{J_{l}} \cdot \Lambda_{J_{(l,j)}},
\Lambda_{J'_{j}}) \\
& = & \varphi_{J}^{-1} \circ (G_{J_{j}} \times G_{J'_{j}}) (x, \Lambda_{J_{l}} \cdot \Lambda_{J_{(l,j)}}, \Lambda_{J'_{j}}) \\
& = & \varphi_{J}^{-1} \circ (G_{J_{l}} \times G_{J'_{l}}) (x,
\Lambda_{J_{l}}, \Lambda_{J_{(l,j)}} \cdot \Lambda_{J'_{j}}) \\
& = & \phi_{l} (x, \Lambda_{J_{l}}, \Lambda_{J_{(l,j)}} \cdot
\Lambda_{J'_{j}}).
\end{eqnarray*}
Since $\phi_{l}$ satisfies (\ref{lemma_collar_1_1}), we have
$\phi_{l} (x, \Lambda_{J_{l}}, \Lambda_{J_{(l,j)}} \cdot
\Lambda_{J'_{j}}) = (x, \Lambda_{J_{l}} \cdot \Lambda_{J_{(l,j)}}
\cdot \Lambda_{J'_{j}})$, or
\begin{eqnarray*}
& & \phi_{j}(x, \lambda_{1}, \cdots, \lambda_{l-1}, 0,
\lambda_{l+1}, \cdots,
\lambda_{j-1}, \lambda_{j+1}, \cdots, \lambda_{n}) \\
& = & (x, \lambda_{1}, \cdots, \lambda_{l-1}, 0, \lambda_{l+1},
\cdots, \lambda_{j-1}, 0, \lambda_{j+1}, \cdots, \lambda_{n}).
\end{eqnarray*}
Define $\theta_{j}: \mathcal{M}_{J} \times [0, \epsilon)^{|J|}
\rightarrow \mathcal{M}_{J} \times [0, \epsilon_{0})^{|J|}$ as
(\ref{lemma_collar_1_2}), we have
\[
\theta_{j}(x, \lambda_{1}, \cdots, \lambda_{l-1}, 0, \lambda_{l+1},
\cdots, \lambda_{n}) = (x, \lambda_{1}, \cdots, \lambda_{l-1}, 0,
\lambda_{l+1}, \cdots, \lambda_{n}).
\]
The operation of $\theta_{j}$ on $\mathcal{M}_{J} \times \prod_{i=1,
i \neq l}^{|J|} [0, \epsilon) \times \{0\}$ is the identity. Thus
\begin{eqnarray*}
& & (\varphi_{J} \circ \theta_{j})^{-1} \circ (G_{J_{l}} \times
G_{J'_{l}}) \\
& = & \theta_{j}^{-1} \circ (\varphi_{J}^{-1} \circ (G_{J_{l}}
\times G_{J'_{l}})) \\
& = & \varphi_{J}^{-1} \circ (G_{J_{l}} \times G_{J'_{l}}) \\
& = & \phi_{l}.
\end{eqnarray*}
So if we modify $\varphi_{J}$ to be $\varphi_{J} \circ \theta_{j}$,
then $\phi_{l}$ ($l<j$) will not change and still satisfy
(\ref{lemma_collar_1_1}). However, $\phi_{j}$ may change and must
satisfy (\ref{lemma_collar_1_1}) now. Thus we get a new
$\varphi_{J}$ such that (\ref{lemma_collar_1_1}) is true for $l \in
\{ 1, \cdots, j \}$.

By repeating this process, we finish the proof of this lemma.
\end{proof}

Now we define $G_{I}$ in $\text{Im} \varphi_{J}$. If $I \npreceq J$, then
$\text{Im} \varphi_{J} \cap \mathcal{M}_{I} = \emptyset$, we don't need to
consider it. We assume $I \preceq J$.

For all $y \in \text{Im} \varphi_{J} \cap \mathcal{M}_{I}$, there exist $x
\in \mathcal{M}_{J}$ and $\Lambda_{J} \in [0,\epsilon)^{n}$ such
that $y = \varphi_{J}(x, \Lambda_{J})$ where $x$ and $\Lambda_{J}$
are unique and $\lambda_{i} = 0$ if and only if $r_{i} \in I$.
Define $G_{I}(y, \Lambda_{I}) = \varphi_{J} (x, \Lambda_{J} +
\Lambda_{I})$. Since $\varphi_{J}$ is a smooth embedding, so is
$G_{I}$. (Actually, if we identify $Im \varphi_{J}$ with
$\mathcal{M}_{J} \times [0, \epsilon)^{|J|}$ via $\varphi_{J}$, then
$G_{I}$ has the form $G_{I}((x, \Lambda_{J}), \Lambda_{I}) = (x,
\Lambda_{J} + \Lambda_{I})$.)

\begin{lemma}\label{lemma_collar_2}
The maps $G_{I}$ satisfy (\ref{theorem_collar_1}) in $Im
\varphi_{J}$.
\end{lemma}
\begin{proof}
Suppose $I_{2} \preceq I_{1} \preceq J$ and $y \in \text{Im} \varphi_{J}
\cap \mathcal{M}_{I_{1}}$, we need to show that $G_{I_{1}}(y,
\Lambda_{I_{1}}) = G_{I_{2}}(G_{I_{1}}(y, \Lambda_{I_{1}}(I_{1} -
I_{2})), \Lambda_{I_{1}, I_{2}})$.

Suppose $y = \varphi_{J}(x, \Lambda_{J})$, we have $G_{I_{1}}(y,
\Lambda_{I_{1}}) = \varphi_{J} (x, \Lambda_{J} + \Lambda_{I_{1}})$,
$G_{I_{1}}(y, \Lambda_{I_{1}}(I_{1} - I_{2})) = \varphi_{J} (x,
\Lambda_{J} + \Lambda_{I_{1}}(I_{1} - I_{2}))$, and
\begin{eqnarray*}
& & G_{I_{2}}(G_{I_{1}}(y, \Lambda_{I_{1}}(I_{1} - I_{2})),
\Lambda_{I_{1}, I_{2}}) \\
& = & G_{I_{2}} (\varphi_{J} (x, \Lambda_{J} +
\Lambda_{I_{1}}(I_{1} - I_{2})), \Lambda_{I_{1}, I_{2}}) \\
& = & \varphi_{J} (x, \Lambda_{J} + \Lambda_{I_{1}}(I_{1} - I_{2}) +
\Lambda_{I_{1}, I_{2}}) \\
& = & \varphi_{J} (x, \Lambda_{J} + \Lambda_{I_{1}}) = G_{I_{1}}(y,
\Lambda_{I_{1}}).
\end{eqnarray*}
This completes the proof of the lemma.
\end{proof}

\begin{lemma}\label{lemma_collar_3}
The maps $G_{I}$ satisfy (\ref{theorem_collar_2}) in $\text{Im}
\varphi_{J}$.
\end{lemma}
\begin{proof}
Suppose $I \preceq J$, $I = I_{1} \cdot I_{2}$, $y_{1} \in
\mathcal{M}_{I_{1}}$, $y_{2} \in \mathcal{M}_{I_{2}}$, and $y_{1}
\cdot y_{2} \in \text{Im} \varphi_{J}$. We need to show that $G_{I}(y_{1}
\cdot y_{2}, \Lambda_{I_{1}} \cdot \Lambda_{I_{2}}) =
(G_{I_{1}}(y_{1}, \Lambda_{I_{1}}), G_{I_{2}}(y_{2},
\Lambda_{I_{2}}))$.

Since $I \preceq J$, we have $J = J_{l} \cdot J'_{l}$, $I_{1}
\preceq J_{l}$ and $I_{2} \preceq J'_{l}$ for some $J_{l} = \{ p,
r_{1}, \cdots, r_{l} \}$ and $J'_{l} = \{ r_{l}, \cdots, r_{n}, q
\}$. Since $y_{1} \cdot y_{2} \in \mathcal{M}_{I_{1}} \times
\mathcal{M}_{I_{2}}$ and $y_{1} \cdot y_{2} = \varphi_{J}(x,
\Lambda_{J})$, we have $x = x_{1} \cdot x_{2}$ for some $x_{1} \in
\mathcal{M}_{J_{l}}$ and $x_{2} \in \mathcal{M}_{J'_{l}}$ and
$\Lambda_{J} = \Lambda_{J_{l}} \cdot \Lambda_{J'_{l}}$ for some
$\Lambda_{J_{l}}$ and $\Lambda_{J'_{l}}$. Thus $y_{1} \cdot y_{2} =
\varphi_{J}(x_{1} \cdot x_{2}, \Lambda_{J_{l}} \cdot
\Lambda_{J'_{l}})$. By Lemma \ref{lemma_collar_1}, $y_{1} =
G_{J_{l}}(x_{1}, \Lambda_{J_{l}})$ and $y_{2} = G_{J'_{l}}(x_{2},
\Lambda_{J'_{l}})$. Furthermore,
\begin{eqnarray*}
& & G_{I} (y_{1} \cdot y_{2}, \Lambda_{I_{1}} \cdot \Lambda_{I_{2}}) \\
& = & \varphi_{J}(x_{1} \cdot x_{2}, \Lambda_{J_{l}} \cdot
\Lambda_{J'_{l}} + \Lambda_{I_{1}} \cdot \Lambda_{I_{2}}) \\
& = & \varphi_{J}(x_{1} \cdot x_{2}, (\Lambda_{J_{l}} +
\Lambda_{I_{1}}) \cdot (\Lambda_{J'_{l}} + \Lambda_{I_{2}})).
\end{eqnarray*}
By Lemma \ref{lemma_collar_1},
\[
\varphi_{J}(x_{1} \cdot x_{2}, (\Lambda_{J_{l}} + \Lambda_{I_{1}})
\cdot (\Lambda_{J'_{l}} + \Lambda_{I_{2}})) = (G_{J_{l}}(x_{1},
\Lambda_{J_{l}} + \Lambda_{I_{1}}), G_{J'_{l}}(x_{2},
\Lambda_{J'_{l}} + \Lambda_{I_{2}})).
\]
Since $|p, r_{l}| < n$ and $|r_{l}, q| < n$, by the outer inductive
hypothesis, $G_{J_{l}}$, $G_{J'_{l}}$, $G_{I_{1}}$ and $G_{I_{2}}$
satisfy (\ref{theorem_collar_1}). Thus
\begin{eqnarray*}
& & (G_{J_{l}}(x_{1}, \Lambda_{J_{l}} + \Lambda_{I_{1}}),
G_{J'_{l}}(x_{2}, \Lambda_{J'_{l}} + \Lambda_{I_{2}})) \\
& = & (G_{I_{1}}(G_{J_{l}}(x_{1}, \Lambda_{J_{l}}),
\Lambda_{I_{1}}),
G_{I_{2}}(G_{J'_{l}}(x_{2}, \Lambda_{J'_{l}}), \Lambda_{I_{2}})) \\
& = & (G_{I_{1}}(y_{1}, \Lambda_{I_{1}}), G_{I_{2}}(y_{2},
\Lambda_{I_{2}})).
\end{eqnarray*}
This completes the proof of the lemma.
\end{proof}

We have defined the desired $G_{I}$ in $\text{Im} \varphi_{J}$ for all $I$
such that $\mathcal{M}_{I} \cap \text{Im} \varphi_{J} \neq \emptyset$.
Clearly, $(\mathcal{M}_{I} \cap \text{Im} \varphi_{J}) \times [0,
\epsilon)^{|I|}$ has a frame $\{ \frac{\partial}{\partial
\lambda_{1}}, \cdots, \frac{\partial}{\partial \lambda_{|I|}} \}$.
Then
\[
\{ \mathcal{N}_{1}(I), \cdots, \mathcal{N}_{|I|}(I) \} = d
G_{I}|_{\Lambda_{J} = 0} \cdot \left \{ \frac{\partial}{\partial
\lambda_{1}}, \cdots, \frac{\partial}{\partial \lambda_{|I|}} \right
\}
\]
serves a desired frame of $A((\mathcal{M}_{I} \cap \text{Im} \varphi_{J}),
\overline{\mathcal{M}(p,q)})$. Identify $\text{Im} \varphi_{J}$ with
$\mathcal{M}_{J} \times [0, \epsilon)^{|J|}$ via $\varphi_{J}$, give
$\text{Im} \varphi_{J}$ the product connection (See Definition
\ref{definition_product_connection} and the comment following it.).
Again, $G_{I}(y, \Lambda_{I}) = \varphi_{J}(x, \Lambda_{J} +
\Lambda_{I})$, and $\Lambda_{J} + t \Lambda_{I}$ for $t \in [0,1]$
is a line segment in $[0, \epsilon)^{|J|}$. Then $G_{I}(y, t
\Lambda_{I})$ is a geodesic segment. Thus $G_{I}(y, \Lambda_{I}) =
\exp (y, \sum_{i=1}^{|I|} \lambda_{i} \mathcal{N}_{i}(I))$ and this
connection is the desired one.

Do the above construction for each $J$ such that $|J|=n$. Clearly,
$G_{J} = \varphi_{J}$ when $|J|=n$. Let $U_{n} = \bigcup_{|J|=n} \text{Im}
G_{J}$, then $U_{n} \supseteq X_{n}$. This completes the first step
of the inner induction.

\textit{(II). The second step of the inner induction (the induction
on $U_{k}$).}

Suppose we have constructed $U_{k+1} = \bigcup_{|I_{0}| \geq k+1} \text{Im}
G_{I_{0}}$. Suppose, for all $I$, we have constructed
$G_{I}|_{U_{k+1}}$, the frames on $\mathcal{M}_{I} \cap U_{k+1}$ and
the connection on $U_{k+1}$ which provides $G_{I}$ via exponential
maps. Moreover, $(\mathcal{M}_{I} \cap U_{k+1}) \times
[0,\epsilon)^{|I|}$ has a product connection if we pull back the
connection on $U_{k+1}$ via $G_{I}$. We shall extend the above
things to those on $U_{k}$.

The construction shares many details with the first step. The
essential point is that the definition of $G_{I}|_{U_{k}}$ should be
an extension of $G_{I}|_{U_{k+1}}$.

Let $U_{k+1}(\delta) = \bigcup_{|I| \geq k+1} G_{I}|_{U_{k+1}}
(\mathcal{M}_{I} \times [0, \delta)^{|I|})$ for $\delta \in (0,
\epsilon)$. It's an open set such that $X_{k+1} \subset
U_{k+1}(\delta) \subset U_{k+1}$. Let $\overline{U_{k+1}(\delta)} =
\bigcup_{|I| \geq k+1} G_{I}|_{U_{k+1}} (\mathcal{M}_{I} \times [0,
\delta]^{|I|})$.

\begin{lemma}\label{lemma_collar_4}
The set $\overline{U_{k+1}(\delta)}$ is closed.
\end{lemma}
\begin{proof}
For each $I_{0}$ such that $|I_{0}| \geq k+1$, we have
$\overline{\mathcal{M}_{I_{0}}} = \bigsqcup_{I_{0} \preceq I}
\mathcal{M}_{I}$ is compact. Moreover, $G_{I_{0}}|_{U_{k+1}}:
\mathcal{M}_{I_{0}} \times [0,\epsilon)^{|I_{0}|} \rightarrow
\overline{\mathcal{M}(p,q)}$ has been defined.

Define $\overline{G_{I_{0}}}: \overline{\mathcal{M}_{I_{0}}} \times
[0,\epsilon)^{|I_{0}|} \rightarrow \overline{\mathcal{M}(p,q)}$ as
$\overline{G_{I_{0}}}(x, \Lambda_{I_{0}}) = G_{I}|_{U_{k+1}}(x,
\Lambda_{I_{0},I})$ for $(x, \Lambda_{I_{0}}) \in \mathcal{M}_{I}
\times [0,\epsilon)^{|I_{0}|}$. Since the maps $G_{I}|_{U_{k+1}}$
satisfy (\ref{theorem_collar_1}), we infer that
$\overline{G_{I_{0}}}$ is well defined and is a smooth embedding.

Thus $\overline{U_{k+1}(\delta)} = \bigcup_{|I_{0}| \geq k+1}
\overline{G_{I_{0}}}(\overline{\mathcal{M}_{I_{0}}} \times [0,
\delta]^{|I_{0}|})$ is compact.
\end{proof}

As the first step, by Lemma \ref{lemma_embed_normal}, for each $J$
such that $|J| = k$, there is a smooth embedding $\varphi_{J}:
\mathcal{M}_{J} \times [0, \epsilon_{0})^{|J|} \rightarrow
\overline{\mathcal{M}(p,q)}$ satisfying the stratum condition. Thus
$d \varphi_{J} \cdot \{ \frac{\partial}{\partial \lambda_{1}}, \cdots,
\frac{\partial}{\partial \lambda_{|J|}} \}$ is a frame satisfying the
stratum condition (See Corollary \ref{corollary_normal_frame}). By
the inner inductive hypothesis, $\mathcal{M}_{J} \cap U_{k+1}$
already has a frame $\{ \mathcal{N}_{1}(J), \cdots,
\mathcal{N}_{|J|}(J) \}$ satisfying the stratum condition. Both
$N_{i}(J)$ and $d \varphi_{J} \frac{\partial}{\partial \lambda_{i}}$
represent nonzero elements in the same $A(\mathcal{M}_{J},
\mathcal{M}_{I}) \cong [0, +\infty)$ for some $I \prec J$ such that
$|I| = |J| - 1$. Thus, for all $\alpha (x) \geq 0$, $\{ \alpha (x)
N_{i}(J) + (1 - \alpha (x)) d \varphi_{J} \frac{\partial}{\partial
\lambda_{i}} \mid i=1, \cdots, n \}$ is also a frame satisfying the
stratum condition. By Lemma \ref{lemma_collar_4} and the partition
of unity,, there is a frame satisfying the stratum condition and
coinciding with the old one in $U_{k+1}(\delta)$ for some $\delta >
0$. Also by the same reason, there is a connection in $U_{k+1} \cup
\text{Im} \varphi_{J}$ such that it coincides with the old one in
$U_{k+1}(\delta)$. Then, by the above frame and connection, we can
modify $\varphi_{J}$ such that it coincides with $G_{J}|_{U_{k+1}}$
in $U_{k+1}(\delta)$. Since $\mathcal{M}_{J} - U_{k+1}(\delta) =
\overline{\mathcal{M}_{J}} - U_{k+1}(\delta)$ is compact, and
$G_{J}|_{U_{k+1}}$ is an embedding, by Lemma \ref{lemma_collar_4},
we infer $\varphi_{J}$ is an embedding defined on $\mathcal{M}_{J}
\times [0, \epsilon_{0})^{|J|}$ for some $\epsilon_{0} \in (0,
\delta]$. Just as the first step, we can modify $\varphi_{J}$
furthermore such that it satisfies the conclusion of Lemma
\ref{lemma_collar_1}. Since originally $\varphi_{J}$ and
$G_{J}|_{U_{k+1}}$ coincide in $U_{k+1}(\delta)$ and
$G_{J}|_{U_{k+1}}$ satisfies (\ref{theorem_collar_2}), the
modification does not change $\varphi_{J}|_{U_{k+1}(\delta)}$. Thus
the modified $\varphi_{J}$ still coincides with $G_{J}|_{U_{k+1}}$
in $U_{k+1}(\delta)$.

The big difference between this step and the first step is as
follows. In the first step, $\text{Im} \varphi_{J}$ are pairwise
disjoint for $|J|=n$. Thus there is no contradiction of the
definition when $G_{I}$ is defined in each $\text{Im} \varphi_{J}$. Now
it's impossible to make $\text{Im} \varphi_{J}$ pairwise disjoint. We
shall control their pair-wise intersections. Suppose $J_{1} \neq
J_{2}$ and $|J_{1}| = |J_{2}| = k$. Then $(\mathcal{M}_{J_{1}} -
U_{k+1}(\delta)) \cap (\mathcal{M}_{J_{2}} - U_{k+1}(\delta))
\subseteq \mathcal{M}_{J_{1}} \cap \mathcal{M}_{J_{2}} = \emptyset$.
Since $\mathcal{M}_{J_{i}} - U_{k+1}(\delta)$ is compact, shrink
$\epsilon_{0}$ if necessary, we have
\[
\varphi_{J_{1}} \left( (\mathcal{M}_{J_{1}} - U_{k+1}(\delta))
\times [0, \epsilon_{0})^{|J_{1}|} \right) \cap \varphi_{J_{2}}
\left( (\mathcal{M}_{J_{2}} - U_{k+1}(\delta)) \times [0,
\epsilon_{0})^{|J_{2}|} \right) = \emptyset.
\]
Since
\[
\varphi_{J_{i}} \left( (\mathcal{M}_{J_{i}} \cap U_{k+1}(\delta))
\times [0, \epsilon_{0})^{|J_{i}|} \right) \subseteq
U_{k+1}(\delta),
\]
we get $\text{Im} \varphi_{J_{1}} \cap \text{Im} \varphi_{J_{2}} \subseteq
U_{k+1}(\delta)$.

Now we define $G_{I}$ in each $\text{Im} \varphi_{J}$. We only need to
consider $I$ such that $I \preceq J$. For all $y \in \mathcal{M}_{I}
\cap \text{Im} \varphi_{J}$, $y = \varphi_{J}(x, \Lambda_{J})$, define
$\widetilde{G}_{I}(J) (y, \Lambda_{I}) = \varphi_{J}(x, \Lambda_{J}
+ \Lambda_{I})$. Given $\varphi_{J} = G_{J}|_{U_{k+1}}$ in
$U_{k+1}(\delta)$, similarly to the argument in the first step, we
get $\widetilde{G}_{I}(J) = G_{I}|_{U_{k+1}}$ in $U_{k+1}(\delta)$.
Since $\text{Im} \varphi_{J_{1}} \cap \text{Im} \varphi_{J_{2}} \subseteq
U_{k+1}(\delta)$, $\widetilde{G}_{I}(J_{1})$ coincides with
$\widetilde{G}_{I}(J_{2})$ in their common domains. Define
$G_{I}|_{\text{Im} \varphi_{J}} = \widetilde{G}_{I}(J)$. Then $G_{I}$ is
well defined on $U_{k+1}(\delta) \cup \bigcup_{|J|=k} \text{Im}
\varphi_{J}$ and it coincides with $G_{I}|_{U_{k+1}}$ in
$U_{k+1}(\delta)$.

Similarly to the first step, the maps $G_{I}|_{\text{Im} \varphi_{J}}$
satisfy (\ref{theorem_collar_1}) and (\ref{theorem_collar_2}).

Shrink $U_{k+1}$ to be $U_{k+1}(\epsilon_{0})$. Again, $G_{J} =
\varphi_{J}$ when $|J|=k$. Let
\[
U_{k} = U_{k+1} \cup \bigcup_{|J|=k} G_{J}(\mathcal{M}_{J} \times
[0,\epsilon_{0})^{k}).
\]
The desired $G_{I}|_{U_{k}}$ is defined in the above. Shrink
$\epsilon_{I}$ to be $\epsilon_{0}$ for all $I$. Give frames to
$\mathcal{M}_{I} \cap U_{k}$ as the first step. For $|J|=k$, give
$\text{Im} G_{J}$ the product connection via $G_{J}$. The old connection in
$U_{k+1}$ is the product connection. Thus the new connection in $\text{Im}
G_{J}$ coincides with the old one in $U_{k+1}$. This completes the
second step of the inner induction.

\textit{(III). The completion of the second step of the outer
induction (the induction on $|p,q|$).}

For the fixed pair $(p,q)$, the construction in $U_{k}$ requires a
shrink of $\epsilon_{I}$ for all $I$ with head $p$ and tail $q$.
However, the inner induction stops in a finite number of steps.
Eventually, we have $\epsilon_{I} > 0$ which are the same for all
$I$ with head $p$ and tail $q$. And if $I_{1} \cdot I_{2} = I$, then
$\epsilon_{I} \leq \epsilon_{I_{i}}$. Thus we have constructed the
desired $G_{I}$ for the pair $(p,q)$ with length $n$. This completes
the second step of the outer induction and also the proof of this
theorem.
\end{proof}

\section{A Byproduct}\label{section_a_byproduct}
The argument for Theorem \ref{theorem_collar} already gives the
following Proposition \ref{proposition_collar} which gives a
compatible collar structure for an arbitrary compact manifold with
faces.

Suppose $L$ is a smooth manifold with faces. Suppose $F_{i}$ ($i=1,
\cdots, n$) are its faces such that $\bigcup_{i=1}^{n} F_{i} =
\bigcup_{k>0} \partial^{k} L$. In other words, $\bigcup_{i=1}^{n}
F_{i}$ is the full boundary of $L$. Suppose the interiors of $F_{i}$
are pairwise disjoint.

Let $I = \{ i_{1}, \cdots, i_{k} \}$ be a subset of $\{ 1, \cdots, n
\}$. Define $|I|=k$. Define $F_{I} = \bigcap_{i \in I} F_{i}$. In
particular, when $I = \emptyset$, define $F_{\emptyset} = L$. Then,
by Lemma \ref{lemma_face_2}, $F_{I}$ is either empty or an $n-k$
dimensional smoothly embedded submanifold with corners insider $L$.
Denote the interior of $F_{I}$ by $F_{I}^{\circ}$.

Let $V_{I} = \prod_{i \in I} [0, +\infty)$ be a factor space of $[0,
+\infty)^{n}$. In other words, $V_{I}$ is the product of the $i$th
coordinate spaces of $[0, +\infty)^{n}$ such that $i \in I$. In
particular, $V_{\emptyset}$ consists of one point. Let
$V_{I}(\epsilon) = \prod_{i \in I} [0, \epsilon)$.

Let $\Lambda_{I} = \{ \lambda_{i_{1}}, \cdots, \lambda_{i_{k}} \}
\in V_{I}$ represent the collaring parameter for $F_{I}^{\circ}$.
Suppose $J \subseteq I$. Define $\Lambda_{I}(I-J) \in V_{I}$ as
\[
  \Lambda_{I}(I - J) (i) =
       \begin{cases}
          0 & i \in J, \\
          \lambda_{i} & i \in I-J.
       \end{cases}
\]
Define $\Lambda_{I,J} \in V_{J}$ as $\Lambda_{I,J}(i) = \lambda_{i}$
for $i \in J$.

\begin{proposition}\label{proposition_collar}
Suppose $L$ is compact. Then collaring maps $G_{I}: F_{I}^{\circ}
\times V_{I}(1) \rightarrow L$ can be defined for all $I$ such that
$F_{I}^{\circ} \neq \emptyset$. These maps satisfy the following
conditions:

(1). They are smooth embeddings which satisfy the following stratum
condition. If $J \subseteq I = \{ i_{1}, \cdots, i_{k} \}$,
$\Lambda_{I} = \{ \lambda_{i_{1}}, \cdots, \lambda_{i_{k}} \} \in
V_{I}(1)$, and $\lambda_{i} = 0$ if and only if $i \in J$, then
$G_{I}(x, \Lambda_{I}) \in F_{J}^{\circ}$ for all $x \in
F_{I}^{\circ}$. In particular, $G_{\emptyset}: F_{\emptyset}^{\circ}
= \partial^{0} L \rightarrow L$ is the inclusion.

(2). They satisfy the following compatibility. If $J \subseteq I$
and $\lambda_{i} > 0$ when $i \notin J$, then, for all $x \in
F_{I}^{\circ}$, we have
\[
G_{I}(x, \Lambda_{I}) = G_{J} (G_{I}(x, \Lambda_{I}(I - J)),
\Lambda_{I,J}).
\]
\end{proposition}

The assumption of Proposition \ref{proposition_collar} is more
general than that of Theorem \ref{theorem_collar} in some sense.
However, this proof is actually even easier than that one because we
only deal with one manifold with faces. It only requires that
(\ref{theorem_collar_1}) is true in a more general setting. We don't
need any more the arguments related to (\ref{theorem_collar_2}) such
as Lemmas \ref{lemma_collar_1} and \ref{lemma_collar_3}. Instead of
a double induction, it suffices to repeat the inner induction in the
proof of Theorem \ref{theorem_collar}. Since there are only finitely
many set $I$, we can find $\epsilon > 0$ such that $\epsilon_{I} =
\epsilon$ for all $I$. By a scaling of parameter, we get $\epsilon =
1$, which finishes the proof.

\section*{Acknowledgements}
I wish to thank Professor Ralph Cohen who told me the importance of face
structures, which improved an earlier version of this paper. I wish
to thank Professor Octav Cornea who encouraged me to publish this paper.
I'm indebted to my Ph.D. advisor Professor John Klein for his direction,
his patient educating, and his continuous encouragement.



\begin{thebibliography}{bib}
\bibitem{barraud_cornea1}J. Barraud and O. Cornea, Homotopic dynamics in
symplectic topology, Morse theoretic methods in nonlinear analysis
and in symplectic topology, 109-148, NATO Sci. Ser. II Math. Phys.
Chem., {\bf 217}, Springer, Dordrecht, 2006.
\bibitem{barraud_cornea2}J. Barraud and O. Cornea, Lagrangian intersections and the
Serre spectral sequence, Ann. of Math., {\bf 166} (2007), 657--722.
\bibitem{burghelea_haller}D. Burghelea and S. Haller, On the topology and analysis
of a closed one form (Novikov's theory revisited), Essays on
geometry and related topics, 133-175, Monogr. Enseign. Math., {\bf
38}, Enseignement Math., Geneva, 2001.
\bibitem{chern_chen_lam}S. Chern, W. Chen and K. Lam, Lectures on
Differential Geometry, Series on University Mathematics, {\bf 1},
World Scientific, 1998.
\bibitem{cohen1}R. Cohen, Floer homotopy theory, realizing chain complexes by module spectra,
and manifolds with corners, Algebraic topology, 39¨C59, Abel Symp.,
{\bf 4}, Springer, Berlin, 2009
\bibitem{cohen2}R. Cohen, The Floer homotopy type of the cotangent
bundle, Pure Appl. Math. Q., {\bf 6} (2010), 391-438.
\bibitem{CJS} R. Cohen, J. Jones, and G. Segal, Morse theory and classifying spaces,
Warwick University preprint, 1995.
\bibitem{cohen_jones_segal2} R. Cohen, J. Jones, and G. Segal,
Floer's infinite dimensional Morse theory and homotopy theory, The
Floer Memorial Volume, 297-325, Progress in Mathematics, {\bf 133},
Birkh\"{a}user Verlag, 1995.
\bibitem{cornea1}O. Cornea, Homotopical dynamics: suspension and duality,
Ergod. Th. \& Dynam. Sys. {\bf 20} (2002), 379-391.
\bibitem{cornea2}O. Cornea, Homotopical dynamics, II: Hopf invariants, smoothings and the Morse complex,
Ann. Sci. ¨¦cole Norm. Sup. (4) {\bf 35} (2002), 549-573.
\bibitem{cornea3}O. Cornea, Homotopical dynamics, III: real sigularities and Hamiltonian
flows, Duke Math. J., {\bf 109} (2001), 183¨C204
\bibitem{cornea4}O. Cornea, Homotopical dynamics, IV: Hopf invariants and Hamiltonian
flows, Comm. Pure Appl. Math. {\bf 55} (2002), 1033-1088.
\bibitem{cornea5}O. Cornea, New obstructions to the thickening of CW-complexes,
Proc. Amer. Math. Soc. {\bf 132} (2004), 2769-2781.
\bibitem{douady}A. Douady, Vari\'{e}t\'{e}s \`{a} bord anguleux et
voisinages tubulaires, S\'{e}minaire Henri Cartan, {\bf 14}
(1961-1962), 1-11.
\bibitem{floer1}A. Floer, Morse theory for Lagranigian
intersections, J. Differential Geom. {\bf 28} (1988), 513-547.
\bibitem{franks}J. Franks, Morse-Smale flows and homotopy theory,
Topology, {\bf 18} (1979), 199-215.
\bibitem{janich}K. J\"{a}nich, On the classification of
$O(n)$-manifolds, Math. Annalen, {\bf 176} (1968), 53-76.
\bibitem{latour}F. Latour, Existence de $1$-formes ferm\'{e}es non singuli\`{e}res dans une classe de cohomologie de de
Rham, Publications math\'{e}matiques de l'I.H.\'{E}.S., {\bf 80}
(1994), 135-194.
\bibitem{palis_de}J. Palis and W. de Melo, Geometric Theorey of
Dynamical Systems, Springer-Verlag, 1982.
\bibitem{qin1}L. Qin, On moduli spaces and CW structures arising from Morse theory on
Hilbert manifolds, Journal of Topology and Analysis, {\bf 2} (2010),
469-526.
\bibitem{qin2}L. Qin, An application of topological equivalence to Morse theory, Preprint,
(2010) (see arXiv:1102.2838).
\bibitem{schwarz}M. Schwarz, Morse Homology, Progress in Mathematics, {\bf
111}, Birkh\"{a}user Verlag, 1993.
\end{thebibliography}
\end{document}